\documentclass[11pt]{amsart}
\usepackage{a4wide,enumerate,color,mathrsfs}
\allowdisplaybreaks

\let\pa\partial  
\let\na\nabla  
\let\eps\varepsilon  
\newcommand{\N}{{\mathbb N}}  
\newcommand{\R}{{\mathbb R}} 
\newcommand{\D}{{\mathscr D}}
\newcommand{\diver}{\operatorname{div}}

\newtheorem{theorem}{Theorem}   
\newtheorem{lemma}[theorem]{Lemma}   
   
\newtheorem{remark}[theorem]{Remark}

\newtheorem{example}{Example} 

\begin{document}  

\title[Degenerate cross-diffusion population models]{Analysis of degenerate 
cross-diffusion population models with volume filling}

\author[N. Zamponi]{Nicola Zamponi}
\address{Institute for Analysis and Scientific Computing, Vienna University of  
Technology, Wiedner Hauptstra\ss e 8--10, 1040 Wien, Austria}
\email{nicola.zamponi@tuwien.ac.at}

\author[A. J\"{u}ngel]{Ansgar J\"{u}ngel}
\address{Institute for Analysis and Scientific Computing, Vienna University of  
Technology, Wiedner Hauptstra\ss e 8--10, 1040 Wien, Austria}
\email{juengel@tuwien.ac.at}

\date{\today}

\thanks{The authors acknowledge partial support from   
the Austrian Science Fund (FWF), grants P24304, P27352, and W1245}  

\begin{abstract}
A class of parabolic cross-diffusion systems modeling the interaction of an
arbitrary number of population species is analyzed in a bounded domain with
no-flux boundary conditions.
The equations are formally derived from a random-walk lattice model 
in the diffusion limit. Compared to previous results in the literature,
the novelty is the combination of general degenerate diffusion and
volume-filling effects. 
Conditions on the nonlinear diffusion coefficients are identified, which yield
a formal gradient-flow or entropy structure. This structure allows for the proof
of global-in-time existence of bounded weak solutions 
and the exponential convergence of the solutions to the constant steady state. 
The existence proof is based on an approximation argument, the entropy inequality,
and new nonlinear Aubin-Lions compactness lemmas. The proof of the large-time
behavior employs the entropy estimate and convex Sobolev inequalities.
Moreover, under simplifiying assumptions on the nonlinearities, 
the uniqueness of weak solutions is shown by using the $H^{-1}$ method,
the $E$-monotonicity technique of Gajewski, and the subadditivity of the Fisher
information.
\end{abstract}



\maketitle


\section{Introduction}\label{sec.intro}

In this paper, we analyze a class of multi-species population cross-diffusion systems
with volume-filling effects.
Such systems arise in various applications, like spatial segregation of interacting
species \cite{SKT79}, chemotactic cell migration in tissues \cite{Pai09}, 
and ion transport through membranes \cite{BSW12}. 
Our model class can be derived from
a system of random-walk master equations in the diffusion limit for a
large class of transition rates (see Appendix \ref{sec.deriv}). The key novelty
of our analysis is the identification of a new entropy or formal gradient-flow
structure and the treatment of non-standard degeneracies in the diffusion 
coefficients, which significantly extends previous results in \cite{Jue14}.

The diffusion systems have the form
\begin{equation}\label{eq}
  \pa_t u - \diver(A(u)\na u) = 0\quad\mbox{in }\Omega,\ t>0,
\end{equation}
with boundary and initial conditions
\begin{equation}\label{bic}
  (A(u)\na u)\cdot\nu = 0 \quad\mbox{on }\pa\Omega,\ t>0, \quad
	u(0)=u^0\quad\mbox{in }\Omega.
\end{equation}
Here, $\Omega\subset\R^d$ ($d\ge 1$) is a bounded domain,
$A(u)=(A_{ij}(u))\in\R^{n\times n}$ is a diffusion matrix,
the function $u=(u_1,\ldots,u_n):\Omega\times(0,\infty)\to\R^n$
is the vector of the proportions of the subpopulations, and 
$u_{n+1}=1-\sum_{i=1}^n u_i$ is the proportion of unoccupied space.
In particular, $0\le u_i\le 1$ for all $i=1,\ldots,n+1$.
The $i$th component of equations \eqref{eq} and \eqref{bic} has to be understood,
respectively, as
$$
  \pa_t u_i - \sum_{j=1}^n\diver(A_{ij}(u)\na u_j) = 0, \quad
	\sum_{j=1}^nA_{ij}(u)\na u_j\cdot\nu = 0.
$$
The boundary condition in \eqref{bic}
means that the physical or biological system is isolated;
the species cannot move through the boundary. 
For ease of presentation, we have neglected reaction and drift terms in the equations. 
We refer to Section \ref{sec.ext} for a discussion of more general models. 

The diffusion matrix in \eqref{eq} is given by
\begin{equation}\label{def.A}
  A_{ij}(u) = \delta_{ij} p_i(u)q_i(u_{n+1}) + u_ip_i(u)q_i'(u_{n+1})
	+ u_iq_i(u_{n+1})\frac{\pa p_i}{\pa u_j}(u), \quad i,j=1,\ldots,n,
\end{equation}
where $\delta_{ij}$ is the Kronecker delta. 
The nonnegative functions $p_i$ and $q_i$ model the transition
rates in the random-walk lattice model. The coefficients $A_{ij}$ are derived
from this model in the diffusion limit (see Section \ref{sec.deriv}).
The function $q_i$ vanishes when the cells are fully packed, i.e.\ if
$\sum_{i=1}^n u_i=1$, so $q_i(0)=0$ and $q_i$ is nondecreasing.
In the literature, several special models were considered and we review
now some of them.

\begin{example}\rm
{\em 1.\ Population-dynamics models.} The case $n=2$, 
$p_i(u)=a_{i0}+a_{i1}u_1+a_{i2}u_2$ and $q_i(u_3)=1$ for $i=1,2$ was suggested
by Shigesada, Kawasaki, and Teramoto \cite{SKT79} to describe the spatial
segregation of interacting populations and to study the
coexistence of two similar species. This model has attracted a lot of attention
in the literature. One of the first existence results is due to Kim \cite{Kim84}
who imposed some restrictions of the parameters $a_{ij}$. The tridiagonal case
$a_{21}=0$ was investigated, e.g., by Amann \cite{Ama89} and Le \cite{Le02}.
The first global existence result without any restriction on the diffusion
coefficients (except positivity) was achieved in \cite{GGJ03} in one space
dimension and in \cite{ChJu04,ChJu06} in several space dimensions.
The case of concave functions $p_1$ and $p_2$ was analyzed by Desvillettes 
et al.\ \cite{DLM14}, recently improved in \cite{DLMT14}. 
The $n$-species case with superlinear functions $p_i(u)$ was
investigated in \cite{Jue14}; also see \cite{BLMP09} for a so-called relaxed system.

{\em 2.\ Ion-transport models.} The case $p_i(u)=1$ for $i=1,\ldots,n$ and 
$q(u_{n+1})=u_{n+1}$ was employed to describe 
the motility of biological cells \cite{SLH09}
or the ion transport through nanopores \cite{BSW12}. The global existence of
bounded weak solutions was proved in \cite{BDPS10}. This result was generalized
in \cite{Jue14} to a class of nondecreasing functions including all
power functions $q(s)=s^\alpha$ with $\alpha\ge 1$. The
models in \cite{BSW12,SLH09} also include a drift term to account
for electric effects, and we discuss these extensions in Section \ref{sec.ext}.

{\em 3.\ Multi-species chemotaxis models.} A special case of the model in \cite{Pai09}
is given by $p_i(u)=1$ and $q(u_{n+1})=u_{n+1}$, similar to the ion-transport
model. In fact, the system in \cite{Pai09} contains additional terms which cannot
be described by \eqref{def.A} since the transition rates assumed in \cite{Pai09}
are not of the type $p_i(u)q_i(u_{n+1})$ (see \eqref{trans} in Appendix 
\ref{sec.deriv}) but they equal $p_i(u)+q_i(u_{n+1})$. We refer to the discussion
in Section \ref{sec.ext}.
\qed
\end{example}

In the model classes (i) and (ii), either $p_i\equiv 1$ or $q_i\equiv 1$. In contrast,
we investigate here a more general model class allowing for nonconstant 
functions $p_i$ {\em and}
$q_i$. A guiding example is system \eqref{eq} with diffusion coefficients
\eqref{def.A} and $p_i(u)=u_1+u_2$, $q_i(s)=s$ for $i=1,2$, which models
volume-filling effects in population systems. The diffusion matrix reads explicitly as 
\begin{equation}\label{ex.A}
  A(u) = \begin{pmatrix}
	u_1(1-u_1-u_2) + (u_1+u_2)(1-u_2) & u_1 \\
	u_2 & u_2(1-u_1-u_2) + (u_1+u_2)(1-u_1)
	\end{pmatrix}.
\end{equation}
We will show in Theorem \ref{thm.ex} that \eqref{eq} with this diffusion 
matrix possesses a global weak solution satisfying $0\le u(t)\le 1$ for all $t>0$.
In fact, Theorem \ref{thm.ex} is concerned with much more general models.

The analysis of system \eqref{eq} with diffusion matrix \eqref{def.A} 
faces a number of
mathematical challenges. First, the equations are {\em strongly coupled} such that
standard tools, like maximum principles and regularity theory, generally
do not apply. Second, the diffusion matrix is generally {\em not positive definite}
and thus, even the local-in-time existence of solutions is nontrivial.
Third, since the variables $u_i$ are proportions, we need to prove
{\em lower and upper bounds} for the solutions (here, $u_i\ge 0$ and 
$\sum_{i=1}^n u_i\le 1$), but maximum principle or invariant region
methods seemingly do not apply.
Fourth, the parabolic system may be {\em degenerate} (e.g.\ like in \eqref{ex.A}
for $u=(0,1)$ or $u=(1,0)$). 

Some of these difficulties have been dealt with in, e.g., \cite{Jue14}
under the assumption that
the diffusion system has a formal entropy or gradient-flow structure, i.e.,
there exists a convex functional $h:\D\to\Omega$ (called entropy density), 
where $\D\subset\R^n$, such that the matrix $B=A(u)h''(u)^{-1}$ 
is positive semi-definite and \eqref{eq} can be written as
\begin{equation}\label{eq.B}
  \pa_t u - \diver(B\na h'(u)) = 0,
\end{equation}
where $h'(u)$ and $h''(u)$ are the Jacobian and Hessian of $h$, respectively.
This formulation has two advantages: First, $H[u]=\int_\Omega h(u)dx$ is a 
Lyapunov functional along solutions $u(t)$ to \eqref{eq}-\eqref{bic},
\begin{equation}\label{dHdt}
  \frac{dH}{dt}[u(t)] = \int_\Omega h'(u)\cdot\pa_t u dx
	= -\int_\Omega \na u:h''(u)A(u)\na u
	= -\int_\Omega \na w:B\na w dx \le 0.
\end{equation}
where $w=h'(u)$ are called entropy variables.
In particular, this yields a gradient-type estimate for $w$ or $u$.
Second, if $h'$ is invertible on $\D$ (see Lemma \ref{lem.h}), 
the original variable $u=(h')^{-1}(w)$
is an element of $\D$. Thus, if $\D$ is a bounded domain, we obtain lower
and upper bounds for $u$ without the use of a maximum principle. In our
situation, we define
$\D = \{u\in\R^n:u_i>0$ for $i=1,\ldots,n$, $\sum_{j=1}^n u_j<1\}$
such that $u_i$ is positive and bounded by one.

There remain still two issues for systems with diffusion coefficients \eqref{def.A}. 
The first one is to identify a suitable entropy
density $h$, the second one is the possible degeneracy. In the example
given by \eqref{ex.A}, we choose
\begin{align*}
  h(u) &= \sum_{i=1}^2 u_i(\log u_i-1) + (1-u_1-u_2)(\log(1-u_1-u_2)-1) \\
	&\phantom{xx}{}+ (u_1+u_2)(\log(u_1+u_2)-1) + 4,
\end{align*}
which yields the matrix
$$
  B = A(u)h''(u)^{-1}	= \begin{pmatrix}
	u_1(u_1+u_2)(1-u_1-u_2) & 0 \\
	0 & u_2(u_1+u_2)(1-u_1-u_2) \end{pmatrix}.
$$
At least one eigenvalue of $B$ vanishes if 
$u\in\pa\D=\{u_1=0$, $u_2=0$, $1-u_1-u_2=0\}$. In this sense, system \eqref{eq}
is called to be of degenerate type.
Generally, systems \eqref{eq} are always of degenerate type since $q(0)=0$. 
Here, we develop a technique to deal with such a degeneracy.

We overcome these issues by developing two main ideas.
Our first key idea is the identification of a class of functions $p_i$ and $q_i$
for which we are able to define a novel entropy density. The second idea is
the extension of the Aubin-Lions compactness lemma to non-standard degenerate cases.
In the following, we detail these concepts.

We make the following structural hypotheses on the functions $p_i$ and $q_i$:
There exist functions $q:[0,1]\to\R$, $\chi:\overline{\D}\to\R$ and a number 
$\gamma>0$ such that for all $i=1,\ldots,n$,
\begin{align}
  & q(s) := q_i(s) > 0, \quad q'(s)\ge\gamma q(s)\mbox{ for }s\in(0,1), 
	\quad q(0)=0, \quad
	q\in C^3([0,1]), \label{hp.q} \\
	& p_i(u) = \exp\left(\frac{\pa\chi(u)}{\pa u_i}\right)\mbox{ for }u\in\D, \quad
	\chi\ge\mbox{ is convex on }\overline{\D}, \quad \chi\in C^3(\overline{\D}).
  \label{hp.p}
\end{align}
Examples of functions $q$ and $p_i$ satisfying these conditions are given
in Remark \ref{rem.ex}. We define the entropy density
\begin{equation}\label{h}
  h(u) = \sum_{i=1}^n(u_i\log u_i-u_i+1) + \int_a^{u_{n+1}}\log q(s)ds + \chi(u),
	\quad u\in\D,
\end{equation}
where $a\in(0,1]$ is such that $\int_a^b\log q(s)ds\ge 0$ for all $b\in(0,1)$,
namely 
\begin{equation}\label{def.a}
  a = \begin{cases}
  1 & \mbox{if }q(1)\leq 1,\\
  q^{-1}(1) & \mbox{if }q(1)>1.
\end{cases} 
\end{equation}
Notice that we require that all functions $q_i$ are the same and
that $p_i$ possesses a particular structure. It seems to be difficult
to treat more general cases, except imposing other conditions.

Surprisingly, system \eqref{eq} with \eqref{def.A} partially decouples in the
entropy variables. Indeed, we may write
 the following formal ``generalized'' gradient-flow formulation
$$
  \pa_t u_i - \diver\left(q(u_{n+1})^2\na \exp\frac{\pa h}{\pa u_i}\right) = 0,
	\quad i=1,\ldots,n,
$$
which makes the degenerate structure more apparent than \eqref{eq}.
We also note that if $q\equiv 1$, we obtain 
$\pa_t u_i = \Delta(\exp(\pa h/\pa u_i))=\Delta(u_ip_i(u))$. 
This structure was exploited in \cite{DLM14,DLMT14}.

A computation, which is made rigorous below, shows that the following
entropy inequality holds:
\begin{equation}\label{1.edi}
  \frac{d}{dt}\int_\Omega h(u)dx
	+ c\int_\Omega\left(q(u_{n+1})^2\sum_{i=1}^n |\na u_i^{1/2}|^2
	+ |\na q(u_{n+1})^{1/2}|^2\right)dx \le 0,
\end{equation}
where $c>0$ is some constant.
We wish to deduce $L^2$ gradient estimates for $u_1,\ldots,u_n$, which are needed
to apply the Aubin-Lions compactness lemma for a suitable approximated system.
However, because of the degeneracy of $q$ (i.e.\ $q(0)=0$), 
these estimates are nontrivial.
We overcome this problem by proving two compactness results. 

The first compactness result essentially states 
that if we have (i) uniform gradient estimates for the bounded sequences $(\xi_\eps)$
and $(\xi_\eps\eta_\eps)$, (ii) a uniform estimate for the (discrete) time derivative
of $\eta_\eps$, and (iii) the strong convergence $\xi_\eps\to\xi$ in $L^2$, 
then up to a subsequence, $\xi_\eps f(\eta_\eps)\to\xi f(\eta)$ in $L^2$
for any continuous function $f$ (Lemma \ref{lem.aubin1}). 
If $\xi_\eps$ were strictly positive, the statement
would be a consequence of the usual Aubin-Lions lemma \cite{Sim87}. Here, we are
able to deal with functions $\xi_\eps$ which may vanish locally. The case
$f(s)=s$ was considered in \cite{BDPS10,Jue14}.

The second compactness result is a generalization of the Aubin-Lions-Dubinski\u{\i} 
lemma; see, e.g., \cite{CJL14,Mou14}. It states that if a bounded sequence
$(u_\eps)$ possesses a uniform estimate for the (discrete) time derivative
and a uniform gradient estimate for $Q(u_\eps)$ and $Q'(u_\eps)$ for some
nonnegative convex increasing function $Q$, then up to a subsequence,
$u_\eps\to u$ strongly in $L^2$ (Lemma \ref{lem.aubin2}). 
This result is complementary to the nonlinear
Aubin-Lions lemma stated in \cite{Mou14} and generalizes the lemma in
\cite{CJL14} stated for $Q(s)=s^\alpha$ with $\alpha>1$.

Based on the above ideas, we prove three results. First, we show the global-in-time
existence of bounded weak solutions to \eqref{eq}-\eqref{def.A} satisfying the
entropy inequality \eqref{1.edi} (Theorem \ref{thm.ex}). 
Second, the entropy inequality and a convex Sobolev
inequality allow us to show that $u_{n+1}(t)$ converges to the constant
steady state in the $L^2$ sense. Moreover, if $q$ is strictly positive, this
convergence also holds for $u_1(t),\ldots,u_n(t)$ (Theorem \ref{thm.conv}). 
Third, if $p_i\equiv 1$ for all $i=1,\ldots,n$, there is a unique weak solution
to \eqref{eq}-\eqref{def.A}. The proof combines the $H^{-1}$ method and the 
$E$-monotonicity technique of Gajewski \cite{Gaj94}. 

The paper is organized as follows. The main results are stated and
commented in Section \ref{sec.main}. Section \ref{sec.aux} is devoted to the
proof of some auxiliary results, like the positive semi-definiteness of
the matrix $h''(u)A(u)$ and the Aubin-Lions compactness lemmas. The three
main theorems are proved in Sections \ref{sec.ex}, \ref{sec.conv}, and 
\ref{sec.unique}, respectively.
Extensions of our model are discussed in Section \ref{sec.ext}.
Appendix A is concerned with the formal derivation of \eqref{eq} from a
random-walk lattice model.


\section{Main results}\label{sec.main}

We state our main theorems and detail the ideas of the proofs.
The first theorem is concerned with the global existence of bounded weak solutions.
Recall that
\begin{equation}\label{D}
  \D = \Big\{u=(u_1,\ldots,u_n)\in\R^n:u_i>0\mbox{ for }i=1,\ldots,n,\
	\sum_{j=1}^n u_j<1\Big\}.
\end{equation}

\begin{theorem}[Global existence]\label{thm.ex}
Let $T>0$, let $u^0:\Omega\to\D$ be a measurable function such that
$h(u^0)\in L^1(\Omega)$,
and let $A(u)$ be given by \eqref{def.A}.
Assume that hypotheses \eqref{hp.q} and \eqref{hp.p} hold. Then: 

\begin{enumerate}[\rm (i)]
\item There exists
a weak solution $u:\Omega\times(0,T)\to\overline{\D}$ to \eqref{eq}-\eqref{bic}
satisfying $u_i\ge 0$, $u_{n+1} := 1-\sum_{i=1}^n u_i\ge 0$, and
\begin{align}
  & q(u_{n+1})^{1/2},\ u_i^{1/2}q(u_{n+1})^{1/2},\ u_ip_i(u)q(u_{n+1})^{1/2}
	\in L^2(0,T;H^1(\Omega)), \label{reg1} \\
	& u_i\in L^\infty(0,T;L^\infty(\Omega)),\quad
	\pa_t u_i\in L^2(0,T;H^1(\Omega)'), \quad i=1,\ldots,n. \label{reg2}
\end{align}
The function $u$ satisfies the weak formulation
\begin{align}\label{weak}
  \sum_{i=1}^n\int_0^T\langle\pa_t u_i,\phi_i\rangle dt
	&+ \sum_{i=1}^n\int_0^T\int_\Omega\big[q(u_{n+1})^{1/2}
	\na\big(u_ip_i(u)q(u_{n+1})^{1/2}\big) \\
	&{}- 3u_ip_i(u)q(u_{n+1})^{1/2}\na q(u_{n+1})^{1/2}\big]\cdot\na\phi_i dxdt = 0
	\nonumber
\end{align}
for all $\phi_1,\ldots,\phi_n\in L^2(0,T;H^1(\Omega))$, and 
$u(0)=u^0$ in the sense of $H^1(\Omega)'$. Here, $\langle\cdot,\cdot\rangle$
denotes the duality product of $H^1(\Omega)'$ and $H^1(\Omega)$. 

\item The following entropy inequality holds:
\begin{align}\label{ei}
  \int_\Omega h(u(t))dx &+ c_0\int_0^t\int_\Omega
	\left(\sum_{i=1}^n q(u_{n+1})^2|\na u_{i}^{1/2}|^2 + |\na q(u_{n+1})^{1/2}|^2\right)
	dxdt \\
	&\le \int_\Omega h(u^0)dx, \nonumber
\end{align}
where $c_0=4p_0\min\{1,\delta\}>0$ with $p_0$ and $\delta$ being defined 
in \eqref{delta} below.

\item If $\int_0^b|\log q(s)|ds=+\infty$ for all $0<b<1$ then $u_{n+1}>0$
a.e.\ in $\Omega\times(0,T)$.
\end{enumerate}
\end{theorem}

\begin{remark}\label{rem.ex}\rm
We present examples of functions $q$ and $p_i$ satisfying \eqref{hp.q} and
\eqref{hp.p}, respectively.
Hypothesis \eqref{hp.q} is satisfied by $q(s)=s^\alpha$ for $s\in[0,1]$,
where $\alpha\ge 1$. Indeed, the inequality $q'(s)\ge\gamma q(s)$ holds for all
$s\in[0,1]$ with $\gamma:=\alpha$.
Another example class is given by $q(s)=\exp(f(s))-1$ with $f(0)=0$ and
$f'(s)\ge\gamma>0$ for $s\in[0,1]$.
A concrete example is $q(s)=\exp(s^\alpha)-1$ with $0<\alpha\le 1$.
A third example is $q(s)=\exp(-s^{-\alpha})$ with $\alpha>0$ which
satisfies the assumption stated in Theorem \ref{thm.ex}, part (iii).

Hypothesis \eqref{hp.p} is satisfied by every function
$p_i(u)=\widetilde p_i(u_i)$, where $\widetilde p_i\in C^1([0,1])$
is strictly positive and nondecreasing. Indeed, let us define
$$
  \chi_i(s) = \int_0^s\log \widetilde p_i(\sigma)d\sigma + k, \quad
	\chi(u) = \sum_{j=1}^n \chi_j(u_j)
$$
for $s\in[0,1]$, $i=1,\ldots,n$, and $u=(u_1,\ldots,u_n)\in\D$. Here,
$k>0$ is such that $\chi_i\ge 0$ in $[0,1]$. Since $\widetilde p_i$ is
strictly positive and nondecreasing in $[0,1]$, it follows that
$\chi''(u)$, given by
$$
  \frac{\pa^2\chi}{\pa u_i\pa u_j}(u) 
	= \delta_{ij}\frac{\widetilde p_i'(u_i)}{\widetilde p_i(u_i)}, \quad
	i,j=1,\ldots,n,
$$
is positive semi-definite and $\chi:\overline{\D}\to[0,\infty)$ is convex. 
Furthermore, $\exp(\pa\chi/\pa u_i)=\widetilde p_i(u_i)=p_i(u)$ for $u\in\D$.

Another example is given by $p_i(u)=(\sum_{j=1}^n a_ju_j)^{a_i}$ with
$a_i\ge 0$, $i=1,\ldots,n$. Indeed, the function $\chi(u)=\sum_{j=1}^n a_ju_j
(\log(\sum_{j=1}^n a_ju_j)-1)$ is convex on $\overline{\D}$ and satisfies
$\exp(\pa\chi/\pa u_i)=\exp(a_i\log(\sum_{j=1}^n a_ju_j))=p_i(u)$.
This example corresponds to the diffusion matrix \eqref{ex.A} for $n=2$
and $a_1=a_2=1$.
\qed
\end{remark}

The proof of Theorem \ref{thm.ex} is based on an approximation and
regularization of \eqref{eq}. More precisely, we consider the semi-discrete
system
$$
  \frac{1}{\tau}(u(w^k)-u(w^{k-1})) = \diver(B(w^k)\na w^k) + \tau^2(\Delta w^k+w^k)
	+ \eps\tau^2\sum_{2\le |\alpha|\le m}(-1)^{|\alpha|-1}D^{2\alpha}w^k
$$
with homogeneous Neumann boundary conditions,
where $\tau>0$, $\eps>0$, $m>d/2$, $u(w)=(h')^{-1}(w)$, $w^k$ approximates $w(k\tau)$,
and $D^{2\alpha}$ is a partial derivative of order $2|\alpha|$, with 
$\alpha\in\N_0^d$ being a multiindex. Compared to \cite{Jue14}, we need {\em two}
regularization levels: the $H^1$ regularization given by $\Delta w^k+w^k$ and
the $H^m$ regularization given by the sum over $\alpha$. The second regularization
is needed to obtain approximate $L^\infty$ solutions (observe that $H^m(\Omega)
\hookrightarrow L^\infty(\Omega)$), while the first one allows us to interpret
the weak formulation in the larger space $H^{-1}$ instead of $H^{-m}$. This
is needed to apply the generalized Aubin-Lions Lemmas \ref{lem.aubin1} and
\ref{lem.aubin2}, for which $H^{-1}$ is required. 

The entropy inequality \eqref{1.edi}, adapted to the above problem, yields uniform
$H^m$ estimates. Hence, applying the Leray-Schauder fixed-point theorem,
we obtain the existence of semi-discrete $H^m$ solutions. 
The same entropy inequality provides a priori estimates uniform in $\tau$ and $\eps$. 
First, we perform the limit $\eps\to 0$,
then the limit $\tau\to 0$. The latter limit is highly nontrivial since
we have only an $L^2$ bound for $q(u_{n+1})\na u_i^{1/2}$, and $q(u_{n+1})=0$ at
$u_{n+1}=0$ is possible. This degeneracy will be overcome by the compactness result
in Lemma \ref{lem.aubin1}.

The second result is about the large-time behavior of the solutions to
the constant steady state given by
$$
  u_i^\infty = \frac{1}{|\Omega|}\int_\Omega u^0_i dx, \ i=1,\ldots,n, \quad
	u_{n+1}^\infty = 1 - \sum_{i=1}^n u_i^\infty.
$$
We are able to prove exponential convergence of $u_{n+1}(t)$ and,
under an additional assumption on $q$, also of $u_1(t),\ldots,u_n(t)$.

\begin{theorem}[Convergence to steady state]\label{thm.conv}
Let $\Omega$ be convex, $u^0\in L^1(\Omega;\D)$, 
let $A(u)$ be given by \eqref{def.A}, and
assume that \eqref{hp.q} and \eqref{hp.p} hold. Furthermore, let
$q\in C^3([0,1])$ be such that $q'$ is strictly positive and $q/q'$ is
concave on $(0,1)$. Let $u:\Omega\times(0,T)\to\D$ be a weak
solution to \eqref{eq}-\eqref{bic} in the sense of Theorem \ref{thm.ex}.
Then 
\begin{equation}\label{conv.unp1}
  \|u_{n+1}(t)-u_{n+1}^\infty\|_{L^2(\Omega)} 
	\le C_1 e^{-\lambda_1 t}, \quad t\ge 0,
\end{equation}
where $C_1 = (2/\gamma)^{1/2}\|h^*(u^0|u^\infty)\|_{L^1(\Omega)}^{1/2}$ and 
$\lambda_1 = c_0q_1/(4c_S)$,
$h^*$ is the relative entropy density (see \eqref{rel.ent}), 
$q_1:=\min_{s\in[0,1]}q'(s)>0$, $c_0>0$ is defined in Theorem \ref{thm.ex}, and 
$c_S>0$ is the constant of the convex Sobolev inequality in Lemma \ref{lem.csi}.
Moreover, if $q_0:=\min_{s\in[0,1]}q(s)>0$, 
\begin{equation}\label{conv.ui}
  \|u_{i}(t)-u_{i}^\infty\|_{L^2(\Omega)} \le C_1 e^{-\lambda_2 t}, \quad t\ge 0,
	\quad i=1,\ldots,n,
\end{equation}
where $\lambda_2=c_0q_0/c_L$ and $c_L>0$ is the constant in the logarithmic
Sobolev inequality (see, e.g., \cite[Lemma~1]{DeFe14}).
\end{theorem}

The convexity of $\Omega$ and the concavity of $q/q'$ is needed to apply 
the convex Sobolev inequality (see Lemma \ref{lem.csi} below). For instance,
$q/q'$ is concave for $q(s)=s^\alpha$ with $\alpha>0$.
The condition on the strict positivity of $q$ contradicts the assumption $q(0)=0$
in Hypothesis \eqref{hp.q}.
However, Theorem \ref{thm.ex} is also valid for functions $q(0)>0$. In fact,
the existence analysis is much easier in this case
since the problem becomes nondegenerate.

The idea of the proof is to derive an inequality for the relative entropy 
\begin{equation}\label{rel.ent}
  \int_\Omega h^*(u|u^\infty)dx 
	= \int_\Omega\big(h(u)-h(u^\infty)-h'(u^\infty)\cdot(u-u^\infty)\big)dx.
\end{equation}
A computation, which is made rigorous in Section \ref{sec.conv}, shows that
$$
  \frac{d}{dt}\int_\Omega\int_{u^0_{n+1}}^{u_{n+1}(t)}\log q(s)dsdx
	+ c\int_\Omega|\na q(u_{n+1})^{1/2}|^2 dx \le 0
$$
for some $c>0$.
The entropy dissipation can be bounded from below (up to a factor) by
the relative entropy by means of 
the convex Sobolev inequality \cite{AMTU01}. Together with the
Gronwall lemma and the convexity of the relative entropy, this
yields exponential convergence of $u_{n+1}(t)$ to $u_{n+1}^\infty$ in the $L^2$ 
norm. In a similar way, we obtain the entropy inequality
$$
  \frac{d}{dt}\int_\Omega\sum_{i=1}^n u_i(t)\log\frac{u_i(t)}{u_i^\infty}dx
	+ c\int_\Omega q(u_{n+1})^2|\na u_i^{1/2}|^2 dx \le 0.
$$
Here, the degeneracy of $q$ at $u_{n+1}=0$ prevents the application of the
logarithmic Sobolev inequality. For this reason, we assume that $q$ is strictly
positive. Then, by Gronwall's lemma again, we deduce the exponential convergence
of $u_i(t)$ to $u_i^\infty$ in the $L^2$ norm.

Our last theorem is a uniqueness result in the special case $p_i\equiv 1$.
This includes the ion-transport model \cite{BSW12}.

\begin{theorem}[Uniqueness of solutions]\label{thm.unique}
Let the assumptions of Theorem \ref{thm.ex} hold and let $p_i\equiv 1$ for 
$i=1,\ldots,n$.
Then there exists a unique weak solution to \eqref{eq}-\eqref{bic} satisfying 
\eqref{reg1}-\eqref{reg2}.
\end{theorem}

The idea of the proof is
to combine the $H^{-1}$ method and the $E$-monotonicity technique of
Gajewski \cite{Gaj94}. In fact, we exploit the special structure of 
\eqref{eq} and \eqref{def.A} in the case $p_i\equiv 1$:
$$
  \pa_t u_i = \diver\big(q(u_{n+1})\na u_i - u_i\na q(u_{n+1})\big), \quad
	i=1,\ldots,n.
$$
Summing all these equations, we end up with a simple equation for $u_{n+1}$:
$$
  \pa_t u_{n+1} = \Delta Q(u_{n+1}), \quad Q'(s) = q(s) + (1-s)q'(s).
$$
The uniqueness for $u_{n+1}$ is shown by the usual $H^{-1}$ method.
The uniqueness for the remaining components $u_i$ is more difficult since
we cannot easily treat the drift term. This is in contrast to the drift-diffusion
equations for semiconductors, 
where a monotonicity property of the drift term can be exploited.
Here, we employ the $E$-monotonicity method \cite{Gaj94}. This method is based
on the convexity of the logarithmic entropy. More precisely, define the
distance
\begin{align*}
  d(u,v) &= \sum_{i=1}^n\int_\Omega\left(\xi(u_i)+\xi(v_i)-2\xi\left(\frac{u_i+v_i}{2}
	\right)\right)dx, \\
	\xi(s) &= s(\log s-1)+1, \quad s\ge 0.
\end{align*}
A formal computation, which is made rigorous in Section \ref{sec.unique}, 
using the subadditivity of the Fisher information (see Lemma \ref{lem.fisher}),
shows that
$$
  \frac{d}{dt}d(u,v) \le 0, \quad t>0,
$$
and consequently, $d(u(t),v(t))\le d(u(0),v(0))=0$ for $t>0$.
Since $\xi$ is convex, we infer that $d(u(t),v(t))\ge 0$, which finally yields
$u_i=v_i$ for $i=1,\ldots,n$. 


\section{Auxiliary results}\label{sec.aux}

\subsection{Invertibility of the entropy transformation}\label{sec.inv}

We show that the transformation of variables $w=h'(u)$ can be inverted.
Recall that the set $\D$ is defined in \eqref{D}.

\begin{lemma} \label{lem.h}
Let assumptions \eqref{hp.q}-\eqref{hp.p} hold. Then the function $h:\D\to\R$,
defined in \eqref{h}, is strictly convex, nonnegative, belongs to $C^2(\D)$,
and its gradient $h':\D\to\R^n$ is invertible.
Moreover, the inverse of the Hessian $h'':\D\to\R^n$ is uniformly bounded.
\end{lemma}

\begin{proof}
We first show that $h':\D\to\R^n$ is invertible. For this, we observe that
$$
  \frac{\pa h}{\pa u_i} = \log u_i - \log q\Big(1-\sum_{j=1}^n u_j\Big)
	+ \frac{\pa\chi}{\pa u_i}, \quad i=1,\ldots,n.
$$
The Jacobian of the function $g=(g_1,\ldots,g_n):\D\to\R^n$, defined by 
$g_i(u)=\log u_i - \log q(1-\sum_{j=1}^n u_j)$, is positive definite since
$$
  \frac{\pa g_i}{\pa u_j} = \frac{\delta_{ij}}{u_i} + \frac{q'(u_{n+1})}{q(u_{n+1})}.
$$
It is shown in Step 1 of the proof of Theorem 6 in \cite{Jue14} that
$g:\D\to\R^n$ is invertible. Thus, we can define the function $f=h'\circ g^{-1}:
\R^n\to\R^n$. Since $h''(u)$ and $g'(u)$ are nonsingular matrices for $u\in\D$,
the Jacobian of $f$,
$$
  f'(y) = h''(g^{-1}(y))(g')^{-1}(g^{-1}(y)),
$$
is nonsingular for $y\in\R^n$. Moreover, by the definitions of $f$ and $g$, we have
\begin{equation}\label{Dh.inv}
  f(y) = y + \chi'(g^{-1}(y)), \quad y\in\R^n.
\end{equation}
Hypothesis \eqref{hp.p} states that $\chi'\in C^0(\overline\D)\subset L^\infty(\D)$,
thus \eqref{Dh.inv} implies that $|f(y)|\to\infty$ as $|y|\to\infty$. 
This property as well as the invertibility of the matrix $f'(u)$ allow us to apply 
Hadamard's global inverse theorem, showing that $f:\R^n\to\R^n$ is invertible.
Consequently, also $h'=f\circ g:\D\to\R^n$ is invertible.

It remains to prove that the inverse of the Hessian of $h$ is bounded.
Since $q'/q\ge 0$, $0<u_i<1$, and $\chi$ is convex in $\D$, the expression 
\begin{equation}\label{D2h}
  \frac{\pa^2 h}{\pa u_i\pa u_j} 
	= \frac{\delta_{ij}}{u_i} + \frac{q'(u_{n+1})}{q(u_{n+1})}
	+ \frac{\pa^2\chi}{\pa u_i\pa u_j}, \quad u\in\D,
\end{equation}
shows that
$v^\top h''(u)v\ge |v|^2$ for all $u\in\D$, $v\in\R^n$. We infer that all points
in the spectrum of $h''$ are strictly positive in $\D$. In particular, $h$ is
strictly convex. As $h''$ is symmetric, we conclude that the inverse of
$h''$ is bounded in $\D$.
\end{proof}


\subsection{Positive definiteness of $HA$}\label{sec.pd}

We show that the product $HA$ of the Hessian $H:=h''(u)$ and the diffusion matrix
$A=A(u)$ is positive definite. This result is needed to deduce gradient estimates
for $u$; see \eqref{dHdt}.

\begin{lemma}\label{lem.HA}
Let assumptions \eqref{hp.q}-\eqref{hp.p} hold. Then the matrix $HA$ is
symmetric and positive definite. More precisely, for all $u\in\D$ and $v\in\R^n$,
we have
\begin{equation}\label{vHAv}
  v^\top(HA)v \ge p_0q(u_{n+1})\sum_{i=1}^n\frac{v_i^2}{u_i}
	+ p_0\delta\frac{q'(u_{u_n+1})^2}{q(u_{n+1})}\left(\sum_{i=1}^n v_i\right)^2,
\end{equation}
where  
\begin{equation}\label{delta}
  p_0=\min_{1\le i\le n}\inf_{u\in\D}p_i(u)>0, \quad
  \delta = \min\left\{\frac12, \frac{2q(1/2)}{\sup_{1/2<s<1}q'(s)}\right\} > 0.
\end{equation}
\end{lemma}

\begin{proof}
First, we verify the symmetry of $HA$. Using \eqref{D2h} and the definition of
$A$, we find that
\begin{align*}
  (HA)_{ij} &= \sum_{k=1}^{n}\left(\frac{\delta_{ik}}{u_{i}} 
	+ \frac{\pa^{2}\chi}{\pa u_{i}\pa u_{k}} + \frac{q'}{q} \right)
	\left(\delta_{kj}p_{k}q + u_{k}p_{k}q' + u_{k}q\frac{\pa p_{k}}{\pa u_{j}}\right) \\
  &= \delta_{ij}\frac{p_{i}q}{u_{i}} + p_{i}q' + \frac{\pa p_{i}}{\pa u_{j}}q 
  + \frac{\pa^{2}\chi}{\pa u_{i}\pa u_{j}}p_{j}q 
	+ \sum_{k=1}^{n}\frac{\pa^{2}\chi}{\pa u_{i}\pa u_{k}}u_{k}p_{k}q' \\
  &\phantom{xx}{} + \sum_{k=1}^{n}\frac{\pa^{2}\chi}{\pa u_{i}\pa u_{k}}
	\frac{\pa p_{k}}{\pa u_{j}}u_{k}q
  + p_{j} q' + \frac{(q')^{2}}{q}\sum_{k=1}^n p_{k}u_{k} 
	+ q'\sum_{k}u_{k}\frac{\pa p_{k}}{\pa u_{j}}.
\end{align*}
Dividing this equation by $q$, defining $\varphi=q'/q$, and 
taking into account that, by assumption \eqref{hp.p},
$$
  \frac{\pa^{2}\chi}{\pa u_{i}\pa u_{j}} 
	= \frac{1}{p_{j}}\frac{\pa p_{j}}{\pa u_{i}} 
	= \frac{1}{p_{i}}\frac{\pa p_{i}}{\pa u_{j}} \quad\mbox{for }i,j=1,\ldots,n,
$$
we infer that
\begin{align}
  \frac{1}{q}(HA)_{ij} 
	&= \delta_{ij}\frac{p_{i}}{u_{i}} + p_{i}\varphi + \frac{\pa p_{i}}{\pa u_{j}}
  + \frac{\pa p_{j}}{\pa u_{i}} + \sum_{k=1}^{n}\frac{\pa p_{k}}{\pa u_{i}}u_{k}
	\varphi \nonumber \\
  &\phantom{xx}{} + \sum_{k=1}^{n}\frac{\pa p_{k}}{\pa u_{i}}
	\frac{\pa p_{k}}{\pa u_{j}} \frac{u_{k}}{p_{k}}
  + p_{j} \varphi + \varphi^{2}\sum_{k=1}^n p_{k}u_{k} 
	+ \varphi\sum_{k=1}^n u_{k}\frac{\pa p_{k}}{\pa u_{j}} \nonumber \\
  &= \delta_{ij}\frac{p_{i}}{u_{i}} + \frac{\pa p_{i}}{\pa u_{j}} 
	+ \frac{\pa p_{j}}{\pa u_{i}} 
	+ \sum_{k=1}^{n}\frac{u_{k}}{p_{k}}\frac{\pa p_{k}}{\pa u_{i}}
	\frac{\pa p_{k}}{\pa u_{j}} \nonumber \\
  &\phantom{xx}{} +\varphi\left( p_{i}+p_{j} + \sum_{k=1}^n u_{k}
	\left(\frac{\pa p_{k}}{\pa u_{i}} + \frac{\pa p_{k}}{\pa u_{j}}\right)\right)
  + \varphi^{2}\sum_{k=1}^n p_{k}u_{k}, \label{comp}
\end{align}
which proves the symmetry of $HA$.

Next, we show the lower bound \eqref{vHAv}. Since $p_i$ is strictly positive
in $\D$, $p_i(u)=\lambda + \widehat p_i(u)$ for any
$\lambda\in(0,p_0)$, where $p_0>0$ is defined in \eqref{delta}, and
$\widehat p_i(u)$ is still strictly positive in $\D$. Then we can write
\eqref{comp} as $HA/q = M+\lambda N$ for two matrices $M=(M_{ij})$ and
$N=(N_{ij})$, defined by
\begin{align*}
  M_{ij} &= \delta_{ij}\frac{\widehat p_{i}}{u_{i}} 
	+ \frac{\pa \widehat p_{i}}{\pa u_{j}} + \frac{\pa \widehat p_{j}}{\pa u_{i}} 
  + \sum_{k=1}^{n}\frac{u_{k}}{\widehat p_{k}+\lambda}
	\frac{\pa \widehat p_{k}}{\pa u_{i}}\frac{\pa \widehat p_{k}}{\pa u_{j}} \\
  &\phantom{xx}{} +\varphi\left(\widehat p_{i} + \widehat p_{j} 
	+ \sum_{k=1}^n u_{k}\left(\frac{\pa \widehat p_{k}}{\pa u_{i}} 
	+ \frac{\pa \widehat p_{k}}{\pa u_{j}}\right)\right)
  + \varphi^{2}\sum_{k=1}^n \widehat p_{k}u_{k}, \\
  N_{ij} &= \frac{\delta_{ij}}{u_{i}} + 2\varphi + \varphi^{2}\sum_{k=1}^n u_k
	= \frac{\delta_{ij}}{u_{i}} + 2\varphi + \varphi^{2}(1-u_{n+1}).
\end{align*}
Let $v\in\R^n$. Then $v^\top (HA/q)v = v^\top Mv + v^\top Nv$.
We consider $v^\top Nv$ first:
\begin{equation}\label{vNv}
  v^\top Nv = \sum_{i=1}^{n} \frac{v_{i}^{2}}{u_{i}} + \varphi(2 
	+ \varphi(1-u_{n+1}))\left(\sum_{i=1}^n v_{i}\right)^{2}.
\end{equation}
The inequalities
\begin{align*}
  & 2 q(s) + (1-s)q'(s) \geq (1-s)q'(s)\geq \frac{1}{2}q'(s) &\mbox{for } 
	0 \leq s \leq\frac{1}{2},\\
  & 2 q(s) + (1-s)q'(s) \geq 2 q(s) 
	\geq \frac{2 q(1/2)}{\sup_{1/2<\sigma<1}q'(\sigma)} q'(s)
  &\mbox{for } \frac{1}{2}\leq s\leq 1,
\end{align*}
imply that
$$
  2q(u_{n+1}) + (1-u_{n+1})q'(u_{n+1}) \ge \delta q'(u_{n+1}),
$$
where $\delta>0$ is defined in \eqref{delta}. Thus, \eqref{vNv} yields
$$
  v^\top Nv \ge \sum_{i=1}^n\frac{v_i^2}{u_i} + \delta\varphi^2\left(\sum_{i=1}^n
	v_i\right)^2.
$$

Finally, we show that $v^\top Mv\ge 0$, which, together with the above estimate
proves the lemma. Using the definition of $M$, we compute
\begin{align}\label{vMv}
  & v^{\top}M v = \sum_{i=1}^n \frac{\widehat p_{i}}{u_{i}}v_{i}^{2} 
	+ \sum_{k=1}^n \frac{u_{k}}{\widehat p_{k}}\left(\sum_{i=1}^n v_{i}
	\frac{\pa \widehat p_{k}}{\pa u_{i}}\right)^{2}
  + 2\sum_{i,j=1}^n \frac{\pa \widehat p_{j}}{\pa u_{i}}v_{i}v_{j} \\
  &\quad {}+2\varphi\left(\sum_{j=1}^n v_{j}\right)\left(\sum_{i=1}^n 
	\widehat p_{i}v_{i} + \sum_{k=1}^n u_{k}\sum_{i=1}^n v_{i}
	\frac{\pa \widehat p_{k}}{\pa u_{i}}\right)
  + \varphi^{2}\left(\sum_{k=1}^n u_{k}\widehat p_{k}\right)
	\left(\sum_{j=1}^n v_{j}\right)^{2}. \nonumber
\end{align}
Let us consider the terms proportional to $\varphi$ and $\varphi^2$:
\begin{align*}
  & 2\varphi\left(\sum_{j=1}^n v_{j}\right)\left(\sum_{i=1}^n \widehat p_{i}v_{i} 
	+ \sum_{k=1}^n u_{k}\sum_{i=1}^n v_{i}\frac{\pa \widehat p_{k}}{\pa u_{i}}\right)
  + \varphi^{2}\left(\sum_{k=1}^n u_{k}\widehat p_{k}\right)
	\left(\sum_{j=1}^n v_{j}\right)^{2} \\
  &\quad = \left(\sum_{k=1}^n u_{k}\widehat p_{k}\right)
	\left[\varphi^{2}\left(\sum_{j=1}^n v_{j}\right)^{2} + 2\varphi\left(
	\sum_{j=1}^n v_{j}\right)\frac{\sum_{i=1}^n \widehat p_{i}v_{i} 
	+ \sum_{k=1}^n u_{k}\sum_{i=1}^n v_{i}
	(\pa \widehat p_{k}/\pa u_{i})}{\sum_{k=1}^n u_{k}\widehat p_{k}}\right] \\
  &\quad = \left(\sum_{k=1}^n u_{k}\widehat p_{k}\right)
	\left[\varphi\sum_{j=1}^n v_{j} + \frac{\sum_{i=1}^n \widehat p_{i}v_{i} 
	+ \sum_{k=1}^n u_{k}\sum_{i=1}^n v_{i}(\pa \widehat p_{k}/\pa u_{i})}{\sum_{k=1}^n 
	u_{k}\widehat p_{k}}\right]^{2} \\
  &\quad\phantom{xx}{} - \frac{\big(\sum_{i=1}^n \widehat p_{i}v_{i} 
	+ \sum_{k=1}^n u_{k}\sum_{i=1}^n v_{i}(\pa \widehat p_{k}/\pa u_{i})
	\big)^{2}}{\sum_{k=1}^n u_{k}\widehat p_{k}}.
\end{align*}
Inserting this expression into \eqref{vMv} yields
\begin{align*}
  v^\top Mv &\ge \sum_{i=1}^n \frac{\widehat p_{i}}{u_{i}}v_{i}^{2} 
	+ \sum_{k=1}^n\frac{u_{k}}{\widehat p_{k}}\left(\sum_{i=1}^n v_{i}
	\frac{\pa \widehat p_{k}}{\pa u_{i}}\right)^{2}
  + 2\sum_{i,j=1}^n\frac{\pa \widehat p_{j}}{\pa u_{i}}v_{i}v_{j} \\
  &\phantom{xx}{} - \frac{\big(\sum_{i=1}^n \widehat p_{i}v_{i} 
	+ \sum_{k=1}^n u_{k}\sum_{i=1}^n v_{i}(\pa \widehat p_{k}/\pa u_{i})
	\big)^{2} }{\sum_{k=1}^nu_{k}\widehat p_{k}}.
\end{align*}
We claim that the right-hand side can be written as a square. To see this,
we introduce the vectors $y=(y_1,\ldots,y_n)$, $z=(z_1,\ldots,z_n)\in\R^n$ by
$$
  y_i = \sqrt{\frac{\widehat p_i}{u_i}}v_i + \sqrt{\frac{u_i}{\widehat p_i}}
	\sum_{k=1}^n v_k\frac{\pa \widehat p_i}{\pa u_k}, \quad
	z_i = \frac{\sqrt{u_i\widehat p_i}}{\sqrt{\sum_{k=1}^n u_k\widehat p_k}},
	\quad i=1,\ldots,n.
$$
The properties
\begin{align*}
  & |z|^{2}=1, \quad |y|^{2} = \sum_{i=1}^n \frac{\widehat p_{i}}{u_{i}}v_{i}^{2} 
	+ \sum_{k=1}^n \frac{u_{k}}{\widehat p_{k}}\left(\sum_{i=1}^n v_{i}
	\frac{\pa \widehat p_{k}}{\pa u_{i}}\right)^{2}
  + 2\sum_{i,j=1}^n \frac{\pa \widehat p_{j}}{\pa u_{i}}v_{i}v_{j}, \\
  & y\cdot z = \frac{\sum_{i=1}^n \widehat p_{i}v_{i} + \sum_{k=1}^n u_{k}
	\sum_{i=1}^n v_{i}(\pa \widehat p_{k}/\pa u_{i}) }{\sqrt{\sum_{k=1}^n 
	u_{k}\widehat p_{k}}}
\end{align*}
show that
$$
  v^\top Mv \ge |y|^2 - (y\cdot z)^2 = |y-(y\cdot z)z|^2 \ge 0.
$$
The lemma is proved.
\end{proof}


\subsection{Generalized Aubin lemmas}\label{sec.aubin}

We prove two generalized Aubin lemmas for functions which are piecewise constant
in time, extending results from \cite{CJL14,Jue14}.

\begin{lemma}[Generalized Aubin lemma I]\label{lem.aubin1}
Let $(\xi^{(\tau)})$, $(\eta^{(\tau)}_1),\ldots,(\eta^{(\tau)}_n)$
be sequences of functions which are piecewise constant in time with constant step 
size $\tau>0$ and which are bounded in $L^\infty(0,T;L^\infty(\Omega))$.
Furthermore, they satisfy the following properties:
\begin{itemize}
\item $\xi^{(\tau)}\to\xi$ strongly in $L^2(0,T;L^2(\Omega))$ as $\tau\to 0$.
\item $\eta^{(\tau)}_i\rightharpoonup\eta_i$ weakly* in 
$L^\infty(0,T;L^\infty(\Omega))$ as $\tau\to 0$ for $i=1,\ldots,n$.
\item There exists $C>0$ such that for all $\tau>0$ and $i=1,\ldots,n$,
\begin{equation}\label{aubin1}
  \|\xi^{(\tau)}\|_{L^{2}(0,T; H^{1}(\Omega))} 
	+  \|\xi^{(\tau)}\eta^{(\tau)}_i\|_{L^{2}(0,T; H^{1}(\Omega))} 
	+ \tau^{-1}\|\eta^{(\tau)}_i - \pi_{\tau}\eta^{(\tau)}_i
	\|_{L^2(\tau,T;H^1(\Omega)')} \le C,
\end{equation}
\end{itemize}
where $\pi_\tau\eta^{(\tau)}_i(\cdot,t)=\eta^{(\tau)}_i(\cdot,t-\tau)$ for
$\tau\le t\le T$ is a shift operator.
Let $D\subset\R^n$ be a compact domain such that $\eta^{(\tau)}(x,t)
=(\eta^{(\tau)}_1,\ldots,\eta^{(\tau)}_n)(x,t)\in D$ 
for a.e.\ $(x,t)\in\Omega\times(0,T)$. Then,
for all $f\in C^0(D;\R^n)$, up to a subsequence, as $\tau\to 0$,
$$
  \xi^{(\tau)}f(\eta^{(\tau)}) \to \xi^{(\tau)}f(\eta^{(\tau)})\quad
	\mbox{strongly in }L^2(0,T;L^2(\Omega)).
$$
\end{lemma}

Since $(\xi^{(\tau)})$ and $(\eta^{(\tau)}_i)$ are assumed to be bounded in
$L^\infty(0,T;L^\infty(\Omega))$, the strong convergence also holds
in $L^p(0,T;L^p(\Omega))$ for all $p<\infty$. This theorem extends
\cite[Lemma 13]{Jue14}, proved for $f(s)=s$, to arbitrary continuous functions $f$.

\begin{proof}
The proof is based on the compactness result in \cite[Lemma 13]{Jue14},
whose proof goes back to \cite{BDPS10}, and an induction and approximation argument.
We perform the proof in two steps. In the first step $f$ is assumed to be a monomial,
in the second step we approximate an arbitrary continuous function by a polynomial
and apply the Stone-Weierstrass theorem. We set $Q_T=\Omega\times(0,T)$.

{\em Step 1.} Let $f(\eta)=\eta^\alpha := \eta_1^{\alpha_1}\cdots\eta_n^{\alpha_n}$,
where $\alpha=(\alpha_1,\ldots,\alpha_n)\in\N_0^n$ is a multiindex. The proof
is an induction argument on the rank $|\alpha|=\sum_{i=1}^n\alpha_i\ge 0$ of the
multiindex. If $|\alpha|=0$, the statement is trivially true. Let us assume that
$\xi^{(\tau)}(\eta^{(\tau)})^\alpha\to \xi\eta^\alpha$ strongly in $L^2(Q_T)$
as $\tau\to 0$ for all $\alpha\in\N_0^n$ with $|\alpha|\le k$, $k\ge 0$.
Let $\alpha\in N_0^n$ be a multiindex such that $|\alpha|=k+1\ge 1$. Then there
exists an index $i_0\in\{1,\ldots,n\}$ such that $\alpha_{i_0}\ge 1$. Hence, we can
define the multiindex $\beta$ such that $\beta_j=\alpha_j - \delta_{i_0,j}$
for $j=1,\ldots,n$ and $|\beta|=k$.

Introduce $y^{(\tau)} = \xi^{(\tau)}(\eta^{(\tau)})^\beta$ and $y=\xi\eta^\beta$.
Clearly, $(y^{(\tau)})$ is bounded in $L^\infty(0,T;L^\infty(\Omega))$.
Since the multiindex $\beta$ has rank $k$ and thus satisfies the induction
assumption, $y^{(\tau)}\to y$ strongly in $L^2(Q_T)$. We claim
that $(y^{(\tau)})$ and $(y^{(\tau)}\eta^{(\tau)}_{i_0})$ are bounded
in $L^2(0,T;$ $H^1(\Omega))$. Indeed, it follows from \eqref{aubin1} that
$\xi^{(\tau)}\na\eta^{(\tau)}_i = \na(\xi^{(\tau)}\eta^{(\tau)}_i)
- \eta^{(\tau)}_i\na\xi^{(\tau)}$ is uniformly bounded in $L^2(Q_T)$.
As a consequence,
\begin{align*}
  \na y^{(\tau)} 
	&= (\eta^{(\tau)})^\beta\na\xi^{(\tau)} 
	+ \xi^{(\tau)}\na(\eta^\beta) \\
	&= (\eta^{(\tau)})^\beta\na\xi^{(\tau)} 
	+ \sum_{k:\beta_k>0}\beta_k(\eta^{(\tau)}_k)^{\beta_k-1}
	\left(\prod_{i\neq k}(\eta^{(\tau)}_i)^{\beta_i}\right)\xi^{(\tau)}
	\na\eta^{(\tau)}_k
\end{align*}
is uniformly bounded in $L^2(Q_T)$, and $(y^{(\tau)})$ is bounded
in $L^2(0,T;H^1(\Omega))$. In a similar way, we can show that 
$(y^{(\tau)}\eta^{(\tau)}_{i_0})$ is bounded in $L^2(0,T;H^1(\Omega))$.
Applying \cite[Lemma 13]{Jue14} to the sequences $(y^{(\tau)})$ and
$(\eta^{(\tau)}_{i_0})$, we infer that there exists a subsequence, which is not
relabeled, such that $y^{(\tau)}\eta^{(\tau)}_{i_0}\to
y\eta_{i_0}$ strongly in $L^2(Q_T)$, which means, by definition
of $y^{(\tau)}$ and $\beta$, that $\xi^{(\tau)}(\eta^{(\tau)})^\beta\to
\xi\eta^\beta$ strongly in $L^2(Q_T)$.

{\em Step 2.} It follows from the previous step that the statement of the lemma
is true if $f$ is a multivariate polynomial. Let $f\in C^0(D;\R^n)$ be given.
Since $D$ is compact, we may apply the Stone-Weierstrass approximation theorem
to obtain, for any $\eps>0$, a multivariate polynomial $P:D\to\R^n$ such that
$|f(\eta)-P(\eta)|<\eps$ for $\eta\in D$. Since $(\xi^{(\tau)})$ and $\xi$ are 
bounded in $L^\infty$, we have for some $C>0$, which does not depend on $\eps$,
$$
  \|\xi^{(\tau)}f(\eta^{(\tau)}) - \xi^{(\tau)}P(\eta^{(\tau)})\|_{L^2(Q_T)}
	\le C\eps, \quad
	\|\xi P(\eta) - \xi f(\eta)\|_{L^2(Q_T)}
	\le C\eps.
$$
Thus,
\begin{align*}
  \|\xi^{(\tau)} f(\eta^{(\tau)}) - \xi f(\eta)\|_{L^2(Q_T)}
	&\le \|\xi^{(\tau)} f(\eta^{(\tau)}) - \xi^{(\tau)} P(\eta^{(\tau)})\|_{L^2(Q_T)} \\
	&\phantom{xx}{}+ \|\xi^{(\tau)}P(\eta^{(\tau)}) - \xi P(\eta)\|_{L^2(Q_T)}
	+ \|\xi P(\eta) - \xi f(\eta)\|_{L^2(Q_T)} \\
	&\le 2C\eps + \|\xi^{(\tau)}P(\eta^{(\tau)}) - \xi P(\eta)\|_{L^2(Q_T)}.
\end{align*}
Since $P$ is a polynomial, the first step of the proof applies and the
last term on the right-hand side converges to zero as $\tau \to 0$ (at least for
a subsequence), resulting in
$$
  \limsup_{\tau\to 0} \|\xi^{(\tau)} f(\eta^{(\tau)}) - \xi f(\eta)\|_{L^2(Q_T)}
	\le 2C\eps.
$$
Since $\eps>0$ is arbitrary and the left-hand side does not depend on $\eps$,
it must vanish, finishing the proof.
\end{proof}

\begin{lemma}[Generalized Aubin lemma II]\label{lem.aubin2}
Let $(\eta^{(\tau)})$ be a sequence of functions which are piecewise constant
in time with constant step size $\tau>0$ and which satisfy $a\le u^{(\tau)}(x,t)\le b$
for a.e.\ $(x,t)\in\Omega\times(0,T)$ for some $a$, $b\in\R$. 
Furthermore, let $Q\in C^1([a,b];\R^n)$ be a nonnegative increasing convex function
and assume that there exists $C>0$ such that for all $\tau>0$,
\begin{align*}
  \|Q(u^{(\tau)})\|_{L^2(0,T;H^1(\Omega))} 
	+ \|Q'(u^{(\tau)})\|_{L^2(0,T;H^1(\Omega))} &\le C, \\ 
	\tau^{-1}\|u^{(\tau)}-\pi_\tau u^{(\tau)}\|_{L^2(\tau,T;H^1(\Omega)')} &\le C.
\end{align*}
Then there exists $u\in L^2(0,T;H^1(\Omega))$ such that, up to a subsequence,
$$
  u^{(\tau)} \to u \quad\mbox{strongly in }L^p(0,T;L^p(\Omega))
	\quad\mbox{for all }p<\infty.
$$
\end{lemma}

This result generalizes Theorem 3a in \cite{CJL14}, stated for $Q(s)=s^m$ with
$m>0$. A related result has been proved in \cite[Theorem 1]{Mou14}. 
Instead of the bound on $Q'(\eta^{(\tau)})$ it is assumed that the function
$|Q'|$ is bounded from below by a positive value near $\pm\infty$ and that
the set $\{x:Q'(x)=0\}$ is finite. Thus, our result seems to be complementary
to that one in \cite{Mou14}.

\begin{proof}
Let $\phi\in X:=H^1(\Omega)\cap L^\infty(\Omega)$ be a test
function. Then the positive and negative parts of $\phi$ satisfy
$\phi_+=\max\{0,\phi\}$, $\phi_-=\min\{0,\phi\}\in X$. By the convexity of $Q$,
we obtain
\begin{align*}
  \frac{1}{\tau} &  \int_{\Omega}(u^{(\tau)} - \pi_{\tau}u^{(\tau)})
	Q'(u^{(\tau)})\phi_{+}dx
	\geq \frac{1}{\tau}\int_{\Omega}\big(Q(u^{(\tau)}) 
	- \pi_{\tau} Q(u^{(\tau)})\big)\phi_{+}dx, \\
  \frac{1}{\tau} & \int_{\Omega}(u^{(\tau)} - \pi_{\tau}u^{(\tau)}) 
	\pi_{\tau}Q'(u^{(\tau)})\phi_{-}dx
  = \frac{1}{\tau}\int_{\Omega}(\pi_{\tau}u^{(\tau)} - u^{(\tau)})
	Q'(\pi_{\tau}u^{(\tau)})(-\phi_{-})dx \\
  &\geq \frac{1}{\tau}\int_{\Omega} 
	\big(\pi_{\tau}Q(u^{(\tau)}) - Q(u^{(\tau)})\big)(-\phi_{-})dx
  = \frac{1}{\tau}\int_{\Omega}\big(Q(u^{(\tau)}) 
	- \pi_{\tau}Q(u^{(\tau)})\big)\phi_{-}dx.
\end{align*}
Adding both inequalities and taking into account that $\phi = \phi_+ + \phi_-$,
we find that
\begin{align}\label{aub1}
  \frac{1}{\tau} & \int_\Omega\big(Q(u^{(\tau)})-\pi_\tau Q(u^{(\tau)})\big)
	\phi dx \\
	&\le \frac{1}{\tau}\int_\Omega(u^{(\tau)}-\pi_\tau u^{(\tau)})
	\big(Q'(u^{(\tau)})\phi_+ + \pi_\tau Q'(u^{(\tau)})\phi_-\big)dx \nonumber \\
	&\le \frac{1}{\tau}\|u^{(\tau)}-\pi_\tau u^{(\tau)}\|_{H^1(\Omega)'}
	\|Q'(u^{(\tau)})\phi_+ + \pi_\tau Q'(u^{(\tau)})\phi_-\|_{H^1(\Omega)}.
	\nonumber
\end{align}
We estimate:
\begin{align*}
  \| & Q'(u^{(\tau)})\phi_+\|_{H^1(\Omega)}^2
	 = \int_\Omega\big(|Q'(u^{(\tau)})|^2\phi_+^2 + |Q'(u^{(\tau)})\na\phi_+
	+ \phi_+\na Q'(u^{(\tau)})|^2\big)dx \\
	&\le 2\int_\Omega|Q'(u^{(\tau)})|^2\big(\phi_+^2+|\na\phi_+|^2\big)dx
	+ 2\int_\Omega\phi_+^2|\na Q'(u^{(\tau)})|^2 dx \\
	&\le C\big(\|\phi\|_{H^1(\Omega)}^2 + \|\phi\|_{L^\infty(\Omega)}^2
	\|Q'(u^{(\tau)})\|_{H^1(\Omega)}^2\big) \\
	&\le C\big(1+\|Q'(u^{(\tau)})\|_{H^1(\Omega)}\big)^2\|\phi\|_X^2.
\end{align*}
In a similar way, we can verify that
$$
  \|\pi_\tau Q'(u^{(\tau)})\phi_-\|_{H^1(\Omega)}
	\le C\big(1+\|\pi_\tau Q'(u^{(\tau)})\|_{H^1(\Omega)}\big)\|\phi\|_X.
$$
Thus, \eqref{aub1} gives
$$
  \frac{1}{\tau}\int_\Omega\big(Q(u^{(\tau)})-\pi_\tau Q(u^{(\tau)})\big)
	\phi dx \le CF^{(\tau)}(t)\|\phi\|_X,
$$
where
$$
  F^{(\tau)}(t) = \frac{1}{\tau}\|(u^{(\tau)}-\pi_\tau u^{(\tau)})(t)\|_{H^1(\Omega)'}
	\big(1 + \|Q'(u^{(\tau)}(t))\|_{H^1(\Omega)}
	+ \|\pi_\tau Q'(u^{(\tau)})\|_{H^1(\Omega)}\big).
$$
This means that
$$
  \tau^{-1}\|(Q(u^{(\tau)})-\pi_\tau Q(u^{(\tau)}))(t)\|_{X'} \le CF^{(\tau)}(t).
$$
The assumptions of the lemma imply that $(F^{(\tau)})$ is bounded in $L^1(\tau,T)$.
Thus, we obtain a uniform estimate for 
$\tau^{-1}(Q(u^{(\tau)})-\pi_\tau Q(u^{(\tau)}))$
in $L^1(\tau,T;X')$. Because of the bound of $Q(u^{(\tau)})$ 
in $L^2(0,T;H^1(\Omega))$
and the compact embedding $H^1(\Omega)\hookrightarrow L^2(\Omega)$, Aubin's lemma
in the version of \cite{DrJu12} yields the existence of a subsequence, which
is not relabeled, such that, as $\tau\to 0$,
$Q(u^{(\tau)})\to Q^*$ strongly in $L^2(0,T;L^2(\Omega))$ and a.e.\ in
$\Omega\times(0,T)$.
Since $Q$ is strictly increasing, this shows that 
$u^{(\tau)}=Q^{-1}(Q(u^{(\tau)}))\to Q^{-1}(Q^*)$ a.e.\ in $\Omega\times(0,T)$.
We set $u:=Q^{-1}(Q^*)$. Then the $L^\infty$ bound for $(u^{(\tau)})$ and
the a.e.\ convergence yield $u^{(\tau)}\to u$ strongly in $L^p(0,T;L^p(\Omega))$
for all $p<\infty$.
\end{proof}


\subsection{Further results}\label{sec.fisher}

We show that the Fisher information 
$\int_\Omega|\na\sqrt{u}|^2 d\mu$ is subadditive, and we recall a convex Sobolev
inequality.

\begin{lemma}\label{lem.fisher}
Let $\mu$ be an absolutely continuous measure with respect to the Lebesque measure,
and let $f$, $g:\Omega\to [0,\infty)$ be measurable, bounded, positive functions such that
$\sqrt{f}$, $\sqrt{g}\in H^1(\Omega,d\mu)$. Then
$$
  \int_\Omega|\na\sqrt{f+g}|^2 d\mu \le \int_\Omega|\na\sqrt{f}|^2 d\mu
	+ \int_\Omega|\na\sqrt{g}|^2 d\mu.
$$
\end{lemma}

This result was proven in \cite[Section 3.6]{PaUn95} in a slightly different
context. For the convenience of the reader, we present the (short) proof.

\begin{proof}
We define the function $F:[0,1]\to\R$ by
$$
  F(s) = \int_\Omega|\na\sqrt{f}|^2 d\mu + \int_\Omega|\na\sqrt{sg}|^2 d\mu
	- \int_\Omega|\na\sqrt{f+sg}|^2 d\mu, \quad s\in[0,1].
$$
Then $F(0)=0$ and $F'(s)\ge 0$ for all $s\in[0,1]$ since
\begin{align*}
  F'(s) &= \int_{\Omega}|\na \sqrt{g}|^{2}d\mu 
	- \int_{\Omega}\na\sqrt{f+sg}\cdot\na\left( \frac{g}{\sqrt{f+sg}} \right)d\mu\\
  &= \int_{\Omega}|\na \sqrt{g}|^{2}d\mu 
	- \int_{\Omega}\na\sqrt{f+sg}\cdot\left(\frac{2\sqrt{g}\na\sqrt{g}}{\sqrt{f+sg}} 
	- \frac{g}{f+sg}\na\sqrt{f+sg}\right)d\mu\\
  &= \int_{\Omega}|\na \sqrt{g}|^{2}d\mu 
	+ \int_{\Omega}\frac{g}{f+sg}|\na \sqrt{f+sg}|^{2}d\mu 
  - 2 \int_{\Omega}\frac{\sqrt{g}}{\sqrt{f+sg}}\na \sqrt{g}\cdot\na\sqrt{f+sg}d\mu\\
  &= \int_{\Omega}\left| \na \sqrt{g} 
	- \frac{\sqrt{g}}{\sqrt{f+sg}}\na\sqrt{f+sg} \right|^{2}d\mu \geq 0.
\end{align*}
We conclude that $F(1)\ge 0$ which shows the lemma.
\end{proof}

\begin{lemma}\label{lem.csi}
Let $\Omega\subset\R^d$ ($d\ge 1$) be a convex domain and let $g\in C^4$ be a
convex function such that $1/g''$ is concave. Then there exists $c_S>0$
such that for all integrable
functions $u$ with integrable $g(u)$ and $g''(u)|\na u|^2$,
$$
  \frac{1}{|\Omega|}\int_\Omega g(u)dx
	- g\bigg(\frac{1}{|\Omega|}\int_\Omega udx\bigg)
	\le \frac{c_S}{|\Omega|}\int_\Omega g''(u)|\na u|^2 dx,
$$
where $|\Omega|$ denotes the measure of $\Omega$.
\end{lemma}

A proof can be found in \cite[Prop. 7.6.1]{BGL14} or \cite[Remark 3.8]{AMTU01}.


\section{Proof of Theorem \ref{thm.ex}}\label{sec.ex}

We divide the proof into several steps.

\subsection{Time discretization and regularization of system \eqref{eq}}

We recall the definition of the entropy variable $w=h'(u)$ for $u\in\D$, where
$h$ is defined in \eqref{h}. Lemma \ref{lem.h} shows that $h'$ is invertible, thus
we may define $u=(h')^{-1}(w)$ for $w\in\R^n$ and we may set $u(w)=u$. 
By Lemma \ref{lem.HA},
the matrix $B(w)=A(u)(h'')^{-1}(u)$ is positive definite for all $w\in\R$ and
$u=u(w)$. We introduce a time discretization for \eqref{eq}. 
Let $T>0$, $N\in\N$, and let $\tau=T/N$ be the time step size.
Furthermore, let $0<\eps<1$ be a regularization parameter and let
$m\in\N$ be such that $H^m(\Omega)\hookrightarrow L^\infty(\Omega)$ 
compactly (i.e.\ choose $m>d/2$). Given $w^{k-1}\in H^m(\Omega;\R^n)$, 
we wish to find $w^k\in H^m(\Omega;\R^n)$ which solves 
the discretized and regularized problem
\begin{equation}\label{dis.reg}
  \frac{1}{\tau}\int_\Omega(u(w^k)-u(w^{k-1}))\cdot\phi dx
	+ \int_\Omega\na\phi:B(w^k)\na w^k dx + \tau^2 b_\eps(\phi,w^k) = 0
\end{equation}
for $\phi\in H^m(\Omega;\R^n)$, where
\begin{equation}\label{beps}
  b_\eps(\phi,w^k) = \int_\Omega(\phi\cdot w^k + \na\phi:\na w^k)dx
	+ \eps\sum_{2\le|\alpha|\le m}D^\alpha\phi\cdot D^\alpha w^k dx,
\end{equation}
and $D^\alpha$ is a partial derivative of order $|\alpha|$.
We prove the existence of weak solutions to \eqref{dis.reg}.

\begin{lemma}\label{lem.dis.reg}
Let \eqref{hp.q}-\eqref{hp.p} hold and let $u^0:\Omega\to\D$ be measurable such
that $h(u^0)\in L^1(\Omega)$. Then there exists a sequence of solutions
$w^k\in H^m(\Omega;\R^n)$ to \eqref{dis.reg} satisfying the discrete entropy
inequality
\begin{equation}\label{a.edi}
  \int_\Omega h(u(w^k))dx + \tau\int_\Omega\na w^k:B(w^k)\na w^k dx
	+ \tau^3 b_\eps(w^k,w^k) \le \int_\Omega h(u(w^{k-1}))dx.
\end{equation}
\end{lemma}

\begin{proof}
The idea is to apply the Leray-Schauder fixed-point theorem.
Let $y\in L^\infty(\Omega;\R^n)$ and $\eta\in[0,1]$ be given. We first solve
the linear problem
\begin{equation}\label{a.a}
  a(w,\phi) = F(\phi)\quad\mbox{for all }\phi\in H^m(\Omega;\R^n),
\end{equation}
where
\begin{align*}
  a(w,\phi) &= \int_\Omega \na\phi:B(y)\na w dx + \tau^2 b_\eps(w,\phi), \\
	F(\phi) &= -\frac{\eta}{\tau}\int_\Omega(u(y)-u(w^{k-1}))\cdot\phi dx.
\end{align*}
The forms $a$ and $F$ are bounded on $H^m(\Omega;\R^n)$. The
matrix $B(y)=A(u(y))h''(u(y))^{-1}$ is positive semi-definite,
$$
  v^\top B(y)v = [h''(u(y))^{-1}v]^\top h''(u(y))A(u(y))[h''(u(y))^{-1} v]\ge 0
$$
for all $v\in\R^n$, thanks to \eqref{vHAv}. Hence, 
the bilinear form $a$ is coercive:
$$
  a(w,w) \ge \eps\tau^2\|w\|_{H^m(\Omega)}^2 \quad\mbox{for }w\in H^m(\Omega;\R^n).
$$
Therefore, we can apply the Lax-Milgram lemma to infer the existence of a unique
solution $w\in H^m(\Omega;\R^n)\hookrightarrow L^\infty(\Omega;\R^n)$ to 
\eqref{a.a}. This defines the fixed-point operator $S:L^\infty(\Omega;\R^n)
\times[0,1] \to L^\infty(\Omega;\R^n)$, $S(y,\eta)=w$, where $w$ solves \eqref{a.a}.

It holds that $S(y,0)=0$ for all $y\in L^\infty(\Omega;\R^n)$. Furthermore,
standard arguments show that $S$ is continuous (see e.g.\
the proof of Lemma 5 in \cite{Jue14}). It remains to prove a uniform bound for
all fixed points $S(\cdot,\eta)$ in $L^\infty(\Omega;\R^n)$. Let 
$w\in L^\infty(\Omega;\R^n)$ be such a fixed point. Then $w$ solves
\eqref{a.a} with $y$ replaced by $w$. With the test function $\phi=w$, we find that
\begin{equation}\label{aux1}
  \frac{\eta}{\tau}\int_\Omega(u(w)-u(w^{k-1}))\cdot w dx
	+ \int_\Omega\na w:B(w)\na wdx + \tau^2 b_\eps(w,w) = 0.
\end{equation}
The convexity of $h$ implies that $h(x)-h(y)\le h'(u)\cdot(x-y)$ for all
$x$, $y\in\D$. Choosing $x=u(w)$ and $y=u(w^{k-1})$ and employing
$h'(u(w))=w$, this gives
$$
  \frac{\eta}{\tau}\int_\Omega(u(w)-u(w^{k-1}))\cdot w dx
	\ge \frac{\eta}{\tau}\int_\Omega \big(h(u(w))-h(u(w^{k-1}))\big)dx.
$$
Taking into account the positive semi-definiteness of $B(w)$, we infer
from \eqref{aux1} that
$$
  \eta\int_\Omega h(u(w))dx + \eps\tau^3\|w\|_{H^m(\Omega)}^2
	\le \eta\int_\Omega h(u(w^{k-1}))dx.
$$
This yields an $H^m$ bound for $w$ uniform in $\eta$ (but not uniform in
$\eps$ and $\tau$). By the Leray-Schauder fixed-point theorem , we conclude
the existence
of a solution $w\in H^m(\Omega;\R^n)$ to \eqref{a.a} with $y$ replaced by
$w$ and $\eta=1$.
\end{proof}

We derive some a priori estimates uniform in $\eps$ and $\tau$. In the following,
we set $u^k=u(w^k)$ for $k\ge 1$, where $(w^k)$ solves \eqref{dis.reg}.

\begin{lemma}\label{lem.edi}
Under the assumptions of Lemma \ref{lem.dis.reg}, there exists a constant
$C>0$ such that for all $\eps$, $\tau>0$,
\begin{align}\label{edi}
  \int_\Omega h(u^k)dx &+ 4\tau p_0\sum_{j=1}^k\int_\Omega q(u_{n+1}^j)\sum_{i=1}^n
	|\na (u_i^j)^{1/2}|^2 dx \\
	&{}+ 4\tau p_0\delta\sum_{j=1}^k\int_\Omega
	|\na q(u_{n+1}^j)^{1/2}|^2 dx + \tau^3\sum_{j=1}^k b_\eps(w^j,w^j)
	\le \int_\Omega h(u^0)dx, \nonumber
\end{align}
where $p_0$ and $\delta$ are defined in \eqref{delta}.
\end{lemma}

\begin{proof}
By Lemma \ref{lem.dis.reg}, the sequence $(w^k)$ satisfies \eqref{dis.reg}.
Then, taking into account the identity 
$\na w^k:B(w^k)\na w^k=\na u^k:h''(u^k)A(u^k)\na u^k$, we deduce that
$$
  \int_\Omega h(u^k)dx + \tau\int_\Omega\na u^k:h''(u^k)A(u^k)\na u^k dx
	+ \tau^3 b_\eps(w^k,w^k) \le \int_\Omega h(u^{k-1})dx.
$$
Resolving this recursion yields
\begin{align*}
  \int_\Omega h(u^k)dx + \sum_{j=1}^k\tau\int_\Omega\na u^j:h''(u^j)A(u^j)\na u^j dx
	+ \tau^3\sum_{j=1}^k b_\eps(w^j,w^j) \le \int_\Omega h(u^0)dx.
\end{align*}
Then the conclusion follows from Lemma \ref{lem.HA} and 
$|\sum_{i=1}^n\na u_i^j|^2=|\na u_{n+1}^j|^2$.
\end{proof}


\subsection{The limit $\eps\to 0$}\label{sec.eps}

Let $(w^k)$ be a sequence of solutions to \eqref{dis.reg}.
We fix $k\in\{1,\ldots,n\}$ and set $u_i^{(\eps)}=u_i^k$ ($i=1,\ldots,n+1$)
and $w_i^{(\eps)}=w_i^k$ ($i=1,\ldots,n$). The identity
\begin{align*}
  (B(w^k)\na w^k)_{i} 
	&= (A(u^k)\na u^k)_{i} \\
	&= q(u^{k}_{n+1})^{1/2}\na \big( u_{i}^{k}p_{i}(u^{k})q(u^{k}_{n+1})^{1/2} \big)
  - 3 u_{i}^{k}p_{i}(u^{k})q(u^{k}_{n+1})^{1/2}\na q(u^{k}_{n+1})^{1/2}
\end{align*}
shows that $u^k$ solves
\begin{align}\label{epsw}
  \frac{1}{\tau}\int_\Omega & (u^k-u^{k-1})\cdot\phi dx 
  + \sum_{i=1}^{n}\int_\Omega \big[q(u^j_{n+1})^{1/2}
	\na \big( u_{i}^j p_{i}(u^j)q(u^j_{n+1})^{1/2} \big) \\
  &{} - 3 u_{i}^j p_{i}(u^j)q(u^j_{n+1})^{1/2}
	\na q(u^j_{n+1})^{1/2}\big]
  \cdot \na\phi_{i} dx + \tau^{2} b_{\eps}(w^j,\phi) = 0 \nonumber
\end{align}
for all $\phi=(\phi_1,\ldots,\phi_n)\in H^m(\Omega;\R^n)$. We wish to pass to
the limit $\eps\to 0$ in \eqref{epsw}.

By Lemma \ref{lem.edi} and definition \eqref{beps} of $b_\eps$, we have
\begin{equation}\label{est.w}
  \eps\tau^3\sum_{j=1}^k\|w^j\|_{H^m(\Omega)}^2 
	+ \tau^3\sum_{j=1}^k\|w^j\|_{H^1(\Omega)}^2 \le C,
\end{equation}
where here and in the following, $C>0$ denotes a generic constant independent of
$\eps$ and $\tau$. 
Thus, because of the boundedness of $(h'')^{-1}$ (see Lemma \ref{lem.h}), 
$$
  \|\na u^{(\eps)}\|_{L^2(\Omega)} 
	= \|(h''(u^{(\eps)}))^{-1}\na w^{(\eps)}\|_{L^2(\Omega)}
	\le C\|\na w^{(\eps)}\|_{L^2(\Omega)} \le C\tau^{-3/2}.
$$
Together with the $L^\infty$ bound for $(u^{(\eps)})$, this implies that
$$
  \|u^{(\eps)}\|_{H^1(\Omega)} \le C\tau^{-3/2}.
$$
Therefore, up to subsequences, as $\eps\to 0$,
\begin{equation*}
  u^{(\eps)}\rightharpoonup u\quad\mbox{weakly in }H^1(\Omega), \quad
	u^{(\eps)}\to u\quad\mbox{strongly in }L^2(\Omega)\mbox{ and a.e. in }\Omega,
\end{equation*}
since $H^1(\Omega)$ embeddes compactly into $L^2(\Omega)$. We infer that
$u_{n+1}^{(\eps)}=1-\sum_{i=1}^n u_i^{(\eps)}\to u_{n+1}:=1-\sum_{i=1}^n u_i$ 
strongly in $L^2(\Omega)$ and a.e.\ in $\Omega$. 
The $L^\infty$ and $H^1$ bounds for $(u^{(\eps)})$ as well as the $L^2$
bound for $\na q(u_{n+1}^{(\eps)})^{1/2}$ in \eqref{edi} show that
\begin{align*}
  \na\big( & u_{i}^{(\eps)} p_{i}(u^{(\eps)}) q(u_{n+1}^{(\eps)})^{1/2})\big) \\ 
	&= u_{i}^{(\eps)} p_{i}(u^{(\eps)}) \na q(u_{n+1}^{(\eps)})^{1/2}
  + q(u_{n+1}^{(\eps)})^{1/2}\sum_{j=1}^{n}\left(
  \delta_{ij}p_{i}(u^{(\eps)}) + u_{i}^{(\eps)}\frac{\pa p_{i}}{\pa u_{j}}(u^{(\eps)})
  \right)\na u^{(\eps)}_{j}
\end{align*}
is uniformly bounded in $L^2(\Omega)$ and hence,
$$
  \|u_{i}^{(\eps)} p_{i}(u^{(\eps)}) q(u_{n+1}^{(\eps)})^{1/2}\|_{H^1(\Omega)}
	\le C\tau^{-1/2}.
$$
We employ the a.e.\ convergence of $(u^{(\eps)})$ and $(u_{n+1}^{(\eps)})$ and the
continuity of $p_i$ and $q$ to obtain
$$
  u_{i}^{(\eps)} p_{i}(u^{(\eps)}) q(u_{n+1}^{(\eps)})^{1/2})
	\to u_ip_i(u)q(u_{n+1})^{1/2} \quad\mbox{a.e. in }\Omega,
$$
and, by the dominated convergence theorem, strongly in $L^2(\Omega)$.
Thus, using the $H^1$ bound,
$$
  u_{i}^{(\eps)} p_{i}(u^{(\eps)}) q(u_{n+1}^{(\eps)})^{1/2}
	\rightharpoonup u_ip_i(u)q(u_{n+1})^{1/2} \quad\mbox{weakly in }H^1(\Omega).
$$
Similar arguments, using the uniform estimates coming from \eqref{edi}, show that
\begin{align}\label{conv.q}
  q(u_{n+1}^{(\eps)})^{1/2}\to q(u_{n+1})^{1/2} &\quad\mbox{strongly in }L^2(\Omega)
  \mbox{ and weakly in }H^1(\Omega), \\
	q(u_{n+1}^{(\eps)})^{1/2}(u_i^{(\eps)})^{1/2}\rightharpoonup
	q(u_{n+1})^{1/2}u_i^{1/2} &\quad\mbox{weakly in }H^1(\Omega). \label{conv.qu}
\end{align}
It follows from the bound \eqref{est.w} that, up to subsequences,
$$
  \eps w^{(\eps)} \to 0\quad\mbox{strongly in }H^m(\Omega), \quad
	w^{(\eps)}\rightharpoonup w\quad\mbox{weakly in }H^1(\Omega).
$$

We set $u^k:=u$. The above convergences holds for all $k=1,\ldots,N$,
where $T=N\tau$. Thus, we obtain a sequence of limit functions $(u^j)$.
The above convergence results are sufficient to pass to the limit $\eps\to 0$
in \eqref{epsw}, resulting in
\begin{align}\label{tauw}
  \frac{1}{\tau} & \int_\Omega (u^k-u^{k-1})\cdot\phi dx 
  + \sum_{i=1}^{n}\int_\Omega \big[q(u^j_{n+1})^{1/2}
	\na \big( u_{i}^j p_{i}(u^j)q(u^j_{n+1})^{1/2} \big) \\
  &{} - 3 u_{i}^jp_{i}(u)q(u^j_{n+1})^{1/2}
	\na q(u_{n+1}^j)^{1/2}\big]
  \cdot \na\phi_{i} dx + \tau^{2}\int_\Omega(w\cdot\phi + \na w:\na\phi)dx 
	= 0 \nonumber
\end{align}
for $\phi\in H^m(\Omega;\R^n)$. By density,
this relation also holds for all $\phi\in H^1(\Omega;\R^n)$.
Note that generally we cannot identify $w$ with $(h')^{-1}(u)$ anymore
but this is not needed in the remaining proof. 

Finally, we wish to pass to the limit $\eps\to 0$ in \eqref{edi}, where
$u^k$ has to be replaced by $u^{(\eps)}$. Since
\begin{equation}\label{aux.q}
  q(u_{n+1}^{(\eps)})^{1/2}\na(u_i^{(\eps)})^{1/2}
	= \na\big(q(u_{n+1}^{(\eps)})^{1/2}(u_i^{(\eps)})^{1/2}\big)
	- (u_i^{(\eps)})^{1/2}\na q(u_{n+1}^{(\eps)})^{1/2},
\end{equation}
the strong convergence $(u_i^{(\eps)})^{1/2}\to u_i^{1/2}$ in $L^4(\Omega)$
and the weak convergences \eqref{conv.q} and \eqref{conv.qu} imply that
\begin{align*}
  q(u_{n+1}^{(\eps)})^{1/2}\na(u_i^{(\eps)})^{1/2}
	&\rightharpoonup \na\big(q(u_{n+1})^{1/2}u_i^{1/2}\big) 
	- u_i^{1/2}\na q(u_{n+1})^{1/2} \\
	&= q(u_{n+1})^{1/2}\na u_i^{1/2} \quad\mbox{weakly in }L^1(\Omega).
\end{align*}
In fact, since by \eqref{edi}, 
$$
  \|q(u_{n+1}^{(\eps)})^{1/2}\na (u_i^{(\eps)})^{1/2}\|_{L^2(\Omega)} 
	\le C\tau^{-1/2},
$$
the above weak convergence also holds in $L^2(\Omega)$. In particular, by the
weak lower semicontinuity of the $L^2$ norm,
\begin{align*}
  \liminf_{\eps\to 0}\int_\Omega q(u_{n+1}^{(\eps)})|\na (u_i^{(\eps)})^{1/2}|^2 dx
	&\ge \int_\Omega q(u_{n+1})|\na u_i^{1/2}|^2 dx, \\
	\liminf_{\eps\to 0}\int_\Omega |\na q(u_{n+1}^{(\eps)})^{1/2}|^2 dx
	&\ge \int_\Omega |\na q(u_{n+1})|^2 dx, \\
	\liminf_{\eps\to 0}\|w^{(\eps)}\|_{H^1(\Omega)}^2
	&\ge \|w\|_{H^1(\Omega)}^2.
\end{align*}
Recall that $u^k=u$ and $w^k=w$. 
Passing to the limit inferior $\eps\to 0$ in \eqref{edi} and observing that
$b_\eps(w^{(\eps)},w^{(\eps)})\ge \|w^{(\eps)}\|_{H^1(\Omega)}^2$, we infer that
\begin{align}\label{edi2}
  \int_\Omega & h(u^k)dx + 4\tau p_0\sum_{j=1}^k\int_\Omega q(u^j_{n+1})\sum_{i=1}^n
	|\na (u_i^j)^{1/2}|^2 dx \\
	&{}+ 4\tau p_0\delta\sum_{j=1}^k\int_\Omega	|\na q(u_{n+1})^{1/2}|^2 dx 
	+ \tau^3\sum_{j=1}^k\|w^j\|_{H^1(\Omega)}^2
  \le \int_\Omega h(u^0)dx. \nonumber
\end{align}


\subsection{The limit $\tau\to 0$}\label{sec.tau}

We set $u^{(\tau)}(x,t) = u^k(x)$ and $w^{(\tau)}(x,t)=w^k(x)$ for $x\in\Omega$,
$t\in((k-1)\tau,k\tau]$. Equation \eqref{tauw} can be formulated as
\begin{align}
  \frac{1}{\tau} & \int_\tau^T\int_\Omega(u^{(\tau)}-\pi_\tau u^{(\tau)})\cdot
	\phi dxdt + \sum_{i=1}^n\int_\tau^T\int_\Omega\big[q(u_{n+1}^{(\tau)})^{1/2}
	\na\big(u_i^{(\tau)}p_i(u^{(\tau)})q(u_{n+1}^{(\tau)})^{1/2}\big) \nonumber \\
	&{}- 3u_i^{(\tau)} p_i(u^{(\tau)})q(u_{n+1}^{(\tau)})^{1/2}\na q(u_{n+1})^{1/2}
	\big]\cdot\na\phi_i dxdt \label{weak.tau} \\
	&{}+ \tau^2\int_\tau^T\int_\Omega(w^{(\tau)}\cdot\phi
	+ \na w^{(\tau)}:\na\phi)dxdt = 0 \nonumber 
\end{align}
for all $\phi(t)\in H^1(\Omega;\R^n)$ being piecewise constant in time and, by
density, for all $\phi\in L^2(0,T;H^1(\Omega))$. Inequality \eqref{edi2} becomes
\begin{align*}
  \int_\Omega & h(u^{(\tau)}(T))dx + 4p_0\int_0^T\int_\Omega
	q(u_{n+1}^{(\tau)})\sum_{i=1}^n|\na(u_i^{(\tau)})^{1/2}|^2 dxdt \\
	&{}+ 4p_0\delta\int_0^T\int_\Omega|\na q(u_{n+1}^{(\tau)})^{1/2}|^2 dxdt
	+ \tau^2\int_0^T\|w^{(\tau)}\|_{H^1(\Omega)}^2 dt \le \int_\Omega h(u^0)dx.
\end{align*}
This gives the following uniform estimates:
\begin{align}
  \|q(u_{n+1}^{(\tau)})^{1/2}\na(u_i^{(\tau)})^{1/2}\|_{L^2(0,T;L^2(\Omega))}
	+ \|q(u_{n+1}^{(\tau)})^{1/2}\|_{L^2(0,T;H^1(\Omega))}&\le C, \label{naq} \\
	\tau\|w^{(\tau)}\|_{L^2(0,T;H^1(\Omega))} &\le C. \label{w.H1}
\end{align}
These bounds as well as the $L^\infty$ bound for $(u_i^{(\tau)})$ show that
\begin{align*}
  \na \big(u_{i}^{(\tau)} & p_{i}(u^{(\tau)}) q(u_{n+1}^{(\tau)})^{1/2}\big)
	= u_{i}^{(\tau)} p_{i}(u^{(\tau)}) \na q(u_{n+1}^{(\tau)})^{1/2} \\
  &\phantom{xx}{}+ q(u_{n+1}^{(\tau)})^{1/2}\sum_{j=1}^{n}\left(
  \delta_{ij}p_{i}(u^{(\tau)}) + u_{i}^{(\tau)}\frac{\pa p_{i}}{\pa u_{j}}(u^{(\tau)})
  \right)\na u^{(\tau)}_{j} \\
  &= u_{i}^{(\tau)} p_{i}(u^{(\tau)}) \na q(u_{n+1}^{(\tau)})^{1/2} \\
  &\phantom{xx}{}+ 2 \sum_{j=1}^{n}(u_{j}^{(\tau)})^{1/2}\left(
  \delta_{ij}p_{i}(u^{(\tau)}) + u_{i}^{(\tau)}\frac{\pa p_{i}}{\pa u_{j}}(u^{(\tau)})
  \right)q(u_{n+1}^{(\tau)})^{1/2}\na(u_{j}^{(\tau)})^{1/2}
\end{align*}
is uniformly bounded in $L^2(0,T;L^2(\Omega))$ and consequently,
\begin{equation}\label{upq.H1}
  \|u_{i}^{(\tau)} p_{i}(u^{(\tau)}) q(u_{n+1}^{(\tau)})^{1/2}
	\|_{L^2(0,T;H^1(\Omega))} \le C.
\end{equation}
Similarly, \eqref{naq} yields the estimate
\begin{equation}\label{uq.H1}
  \|(u_i^{(\tau)})^{1/2}q(u_{n+1}^{(\tau)})^{1/2}\|_{L^2(0,T;H^1(\Omega))} \le C.
\end{equation}
Thus, the $L^\infty$ bound on $(u_i^{(\tau)})$ and estimates 
\eqref{naq} and \eqref{w.H1} give
\begin{align}
  \tau^{-1} & \|u^{(\tau)}-\pi_\tau u^{(\tau)}\|_{L^2(\tau,T;H^1(\Omega)')} 
	\nonumber \\
	&\le \sum_{i=1}^n\|q(u_{n+1}^{(\tau)})^{1/2}\|_{L^\infty(\tau,T;L^\infty(\Omega))}
	\big\|\na\big(u_i^{(\tau)} p_i(u^{(\tau)})
	q(u_{n+1}^{(\tau)})^{1/2}\big)\big\|_{L^2(\tau,T;L^2(\Omega))} \label{hyp.a1} \\
	&\phantom{xx}{}+ 3\sum_{i=1}^n\|u_i^{(\tau)}p_i(u^{(\tau)})q(u_{n+1})^{1/2}
	\|_{L^\infty(\tau,T;L^\infty(\Omega))}\|\na q(u_{n+1})^{1/2}
	\|_{L^2(0,T;L^2(\Omega))} \nonumber \\
	&\phantom{xx}{}+ \tau^2\|w^{(\tau)}\|_{L^2(\tau,T;H^1(\Omega))}^2
	\le C. \nonumber
\end{align}

Now, we define the function $Q(s) = \int_0^s q(\sigma)^{1/2}d\sigma$ for
$s\in[0,1]$. Then $Q\in C^1([0,1])$ is nonnegative, convex, and strictly increasing.
It holds (see \eqref{naq})
\begin{equation}\label{hyp.a2}
  \|Q'(u_{n+1}^{(\tau)})\|_{L^2(0,T;H^1(\Omega))} \le C.
\end{equation}
By assumption \eqref{hp.q}, $q(u_{n+1}^{(\tau)})/q'(u_{n+1}^{(\tau)})$ is 
uniformly bounded a.e.\ and thus,
$$
  \na Q(u_{n+1}^{(\tau)}) = \frac{Q'(u_{n+1}^{(\tau)})}{Q''(u_{n+1}^{(\tau)})}
	\na Q'(u_{n+1}^{(\tau)}) 
	= \frac{2q(u_{n+1}^{(\tau)})}{q'(u_{n+1}^{(\tau)})}\na Q'(u_{n+1}^{(\tau)})
$$
is uniformly bounded in $L^2(0,T;L^2(\Omega))$. We conclude that
\begin{equation}\label{hyp.a3}
  \|Q(u_{n+1}^{(\tau)})\|_{L^2(0,T;H^1(\Omega))} \le C.
\end{equation}
Estimates \eqref{hyp.a1}-\eqref{hyp.a3} show that the assumptions of Lemma 
\ref{lem.aubin2} are fulfilled, and we infer the existence of a subsequence,
which is not relabeled, such that, as $\tau\to 0$,
\begin{equation}\label{conv.u}
  u_{n+1}^{(\tau)} \to u_{n+1}\quad\mbox{strongly in }L^r(0,T;L^r(\Omega)),
	\quad r<\infty.
\end{equation}
This result, the bound \eqref{naq}, and the continuity of $q$ imply that
\begin{align}
  q(u_{n+1}^{(\tau)})^{1/2} \to q(u_{n+1})^{1/2} &\quad\mbox{strongly in }
	L^r(0,T;L^r(\Omega)), \quad r<\infty, \label{conv.q12} \\
  q(u_{n+1}^{(\tau)})^{1/2} \rightharpoonup q(u_{n+1})^{1/2} &\quad\mbox{weakly in }
	L^2(0,T;H^1(\Omega)). \label{conv.qH1}
\end{align}

Using the $L^\infty$ bound for $(u_i^{(\tau)})$, we have, up to a subsequence,
$u_i^{(\tau)}\rightharpoonup^* u_i$ weakly$^*$ in $L^\infty(0,T;L^\infty(\Omega))$
as $\tau\to 0$. This convergence also holds in $L^2$. Thus, \eqref{conv.u}
implies that the relation $u_{n+1}^{(\tau)}=1-\sum_{i=1}^n u_i^{(\tau)}$ 
is satisfied by the limit function, $u_{n+1}=1-\sum_{i=1}^n u_i$. The set
$\{v\in L^2(0,T;L^2(\Omega)):v\ge 0$ a.e.\ in $\Omega\times(0,T)\}$
is (strongly) closed and convex. Hence, it is also weakly closed, and the
property $u_i^{(\tau)}\ge 0$ holds in the limit, i.e.\ 
$u_i\ge 0$ a.e.\ in $\Omega\times(0,T)$.

We turn to the convergence properties of the sequences $(u_i^{(\tau)})$ for 
$i=1,\ldots,n$. We cannot expect strong convergence of $(u_i^{(\tau)})$,
but the generalized Aubin-Lions Lemma \ref{lem.aubin1} shows that the product
$f(u^{(\tau)})q(u_{n+1}^{(\tau)})^{1/2}$ converges strongly, where $f$
is any continuous function.
To make this precise, we verify the assumptions of Lemma \ref{lem.aubin1}.
Set $\xi^{(\tau)}:=q(u_{n+1}^{(\tau)})^{1/2}$ and $\eta_i^{(\tau)}:=u_i^{(\tau)}$.
Because of the $L^\infty$ bounds for $(u_i^{(\tau)})$, up to a subsequence,
$$
  \eta_i^{(\tau)}\rightharpoonup^* \eta_i = u_i\quad\mbox{weakly$^*$ in }
	L^\infty(0,T;L^\infty(\Omega)).
$$
Furthermore, by \eqref{conv.q12},
$\xi^{(\tau)}\to\xi=q(u_{n+1})^{1/2}$ strongly in $L^2(0,T;L^2(\Omega))$.
Estimates \eqref{naq}, \eqref{uq.H1}, and \eqref{hyp.a1} show that
the assumptions of Lemma \ref{lem.aubin1} are satisfied, and we conclude the
existence of a subsequence (not relabeled) such that
$$
  f(u^{(\tau)})q(u_{n+1}^{(\tau)})^{1/2} = f(\eta_i^{(\tau)})\xi^{(\tau)}
	\to f(\eta)\xi = f(u_i)q(u_{n+1})^{1/2} \quad\mbox{strongly in }L^2(0,T;L^2(\Omega))
$$
for any function $f\in C^0(\overline{\D};\R^n)$. We choose $f(s)=s_i^{1/2}$ and
$f(s)=s_i p_i(s)$ for $s=(s_i)\in\overline{\D}$. Then
\begin{align}
  (u_i^{(\tau)})^{1/2}q(u_{n+1}^{(\tau)})^{1/2} \to
	u_i^{1/2}q(u_{n+1})^{1/2} &\quad\mbox{strongly in }L^2(0,T;L^2(\Omega)), 
	\nonumber \\
	u_i^{(\tau)} p_i(u^{(\tau)})q(u_{n+1}^{(\tau)})^{1/2}
	\to u_ip_i(u)q(u_{n+1})^{1/2} &\quad\mbox{strongly in }L^2(0,T;L^2(\Omega)).
	\label{conv.upqL2}
\end{align}
We conclude from the bounds \eqref{upq.H1} and \eqref{uq.H1} that the above
sequences converge weakly in $L^2(0,T;$ $H^1(\Omega))$ and the limit functions
can be identified:
\begin{align}
  (u_i^{(\tau)})^{1/2}q(u_{n+1}^{(\tau)})^{1/2} \rightharpoonup
	u_i^{1/2}q(u_{n+1})^{1/2} &\quad\mbox{weakly in }L^2(0,T;H^1(\Omega)), 
	\label{conv.uq} \\
	u_i^{(\tau)} p_i(u^{(\tau)})q(u_{n+1}^{(\tau)})^{1/2}
	\rightharpoonup u_ip_i(u)q(u_{n+1})^{1/2} 
	&\quad\mbox{weakly in }L^2(0,T;H^1(\Omega)). \label{conv.upq}
\end{align}
We infer from estimate \eqref{hyp.a1} that
$$
  \tau^{-1}(u_i^{(\tau)}-\pi_\tau u_i^{(\tau)}) \rightharpoonup
	\pa_t u_i \quad\mbox{weakly in }L^2(0,T;H^1(\Omega)'), \quad i=1,\ldots,n.
$$
Moreover, taking into account \eqref{w.H1}, 
$$
  \tau^2 w^{(\tau)} \to 0 \quad\mbox{strongly in }L^2(0,T;H^1(\Omega)).
$$
These convergence results as well as the convergences 
\eqref{conv.q12}-\eqref{conv.upqL2} and \eqref{conv.upq} allow us to
perform the limit $\tau\to 0$ in \eqref{weak.tau},
which yields the weak formulation \eqref{weak}.


\subsection{Entropy inequality and positivity}

It remains to verify the entropy inequality \eqref{ei} and the (conditional) positivity
of $u_{n+1}$. Since the entropy density $h$ is convex and continuous,
it is weakly lower semi-continuous \cite[Corollary 3.9]{Bre11}. Thus, by
the weak convergence of $(u_i^{(\tau)}(t))$,
$$
  \int_\Omega h(u(t))dx \le \liminf_{\tau\to 0}\int_\Omega h(u^{(\tau)}(t))dx
	\quad\mbox{for a.e. }t>0.
$$
Employing the convergences \eqref{conv.u}, \eqref{conv.q12}, and \eqref{conv.uq},
it follows that
$$
  q(u_{n+1}^{(\tau)})\na (u_{i}^{(\tau)})^{1/2} 
	= q(u_{n+1}^{(\tau)})^{1/2}\na \big(q(u_{n+1}^{(\tau)})^{1/2}
	(u_{i}^{(\tau)})^{1/2}\big)
  - q(u_{n+1}^{(\tau)})^{1/2}(u_{i}^{(\tau)})^{1/2}\na q(u_{n+1}^{(\tau)})^{1/2}
$$
converges weakly in $L^1$, but because of the $L^2$ bound \eqref{naq} this
convergence also holds in $L^2$:
$$
  q(u_{n+1}^{(\tau)})\na (u_{i}^{(\tau)})^{1/2} \rightharpoonup
	q(u_{n+1})\na u_i^{1/2} \quad\mbox{weakly in }L^2(0,T; L^2(\Omega)).
$$
These results, together with \eqref{conv.qH1}, allow us to pass to the limit
inferior $\tau\to 0$ in \eqref{edi2}, yielding \eqref{ei}.

Finally, assume that 
\begin{equation}\label{q.infty}
  \int_0^b|\log q(s)|ds=+\infty \quad\mbox{for all }0<b<1.
\end{equation}
We deduce from the discrete entropy inequality \eqref{edi2} and 
definition \eqref{h} of $h$ that
$$
  \int_\Omega\int_a^{u^{(\tau)}_{n+1}(x,t)}\log q(s)dsdx
	\le \int_\Omega h(u^{(\tau)}(x,t))dx \le \int_\Omega h(u^0)dx
	\quad\mbox{for a.e. }t>0.
$$
Then, by the strong convergence \eqref{conv.u} of $(u_{n+1}^{(\tau)})$
and the nonnegativity of $\int_a^b \log q(s)ds\ge 0$, we can apply Fatou's lemma
yielding
$$
  \int_\Omega\int_a^{u_{n+1}(x,t)}\log q(s)dsdx \le \int_\Omega h(u^0)dx.
$$
In particular, $\int_a^{u_{n+1}(x,t)}\log q(s)ds<\infty$ for a.e.\ $x\in\Omega$.
We conclude from this fact and assumption \eqref{q.infty} that $u_{n+1}(x,t)>0$
for a.e.\ $x\in\Omega$ and $t\in(0,T)$, which ends the proof.


\section{Proof of Theorem \ref{thm.conv}}\label{sec.conv}

We define the relative entropy density
\begin{equation}\label{hstar}
  h^*(u|u^\infty) = h(u)-h(u^\infty)-h'(u^\infty)\cdot(u-u^\infty)\quad
	\mbox{for }u\in\R^n.
\end{equation}
We split $h^*$ in several parts, $h^*=h_1^*+h_2^*+h_3^*$, 
each of which is nonnegative, where
\begin{align*}
  h_1^*(u|u^\infty) &= \sum_{i=1}^n\left( u_i\log\frac{u_i}{u_i^\infty}-u_i+u_i^\infty
	\right), \\
	h_2^*(u_{n+1}|u^\infty) &= \int_{u_{n+1}^\infty}^{u_{n+1}}
	\log\frac{q(s)}{q(u_{n+1}^\infty)}ds 
	= \int_1^{u_{n+1}/u_{n+1}^\infty}
	\log\frac{q(\sigma u_{n+1}^\infty)}{q(u_{n+1}^\infty)}u_{n+1}^\infty d\sigma, \\
	h_3^*(u|u^\infty) &= \chi(u)-\chi(u^\infty)-\sum_{i=1}^n(u_i-u_i^\infty)
	\log p_i(u^\infty),
\end{align*}
where $\chi$ is defined in \eqref{hp.p}.
The entropy inequality \eqref{ei} and the $L^1$ conservation of $u(t)$ give
\begin{align}\label{edi3}
  \int_\Omega h^*(u(t)|u^\infty)dx
	&+ c_0\int_0^t\int_\Omega\Big(q(u_{n+1})^2\sum_{i=1}^n|\na u_i^{1/2}|^2
	+ |\na q(u_{n+1})^{1/2}|^2\Big)dxds \\
	&\le \int_\Omega h^*(u^0|u^\infty)dx, \quad t>0. \nonumber
\end{align}
We prove now that the above entropy inequality, reduced to an inequality for
$h_2^*$, and the convex Sobolev inequality in
Lemma \ref{lem.csi} yield exponential convergence of $u_{n+1}(t)$, while
the entropy estimate for $h_1^*$ and the logarithmic Sobolev inequality
allows us to conclude the convergence of $u_i(t)$ for $i=1,\ldots,n$.

{\em Step 1: Exponential convergence of $u_{n+1}(t)$.}
Let $g(s) = \int_1^s\log q(\sigma u_{n+1}^\infty)d\sigma$
for $s\in[0,1]$. This function is convex since $g''(s)=u_{n+1}^\infty
q'(su_{n+1}^\infty)/q(su_{n+1}^\infty)>0$ by assumption. Again by assumption,
$1/g''=(u_{n+1}^\infty)^{-1} q/q'$ is concave. 
Choosing $\phi_i=1$ in the weak formulation \eqref{weak} and summing the
equations from $i=1,\ldots,n$, it follows that
$\int_\Omega u_{n+1}(t)/u_{n+1}^\infty dx=\int_\Omega u_{n+1}^0/u_{n+1}^\infty dx 
= |\Omega|$ for $t>0$, and in particular,
$$ 
  g\left(\frac{1}{|\Omega|}\int_\Omega \frac{u_{n+1}}{u_{n+1}^\infty}dx\right) 
	= g(1) = 0.
$$
Thus, we may apply 
the convex Sobolev inequality in the version of Lemma \ref{lem.csi}:
\begin{align*}
  \frac{1}{|\Omega|}\int_\Omega h_2^*(u_{n+1} | u^\infty)dx
	&= \frac{u_{n+1}^\infty}{|\Omega|}\int_\Omega 
	g\left(\frac{u_{n+1}}{u_{n+1}^\infty}\right)dx
	\le \frac{c_S u_{n+1}^\infty}{|\Omega|}\int_\Omega 
	g''\left(\frac{u_{n+1}}{u_{n+1}^\infty}\right)
	\left|\na\frac{u_{n+1}}{u_{n+1}^\infty}\right|^2 dx \\
	&= \frac{c_S}{|\Omega|}\int_\Omega \frac{q'(u_{n+1})}{q(u_{n+1})}
	|\na u_{n+1}|^2 dx.
\end{align*}
By assumption, $q'$ is strictly positive on $[0,1]$, i.e.\
$0<q_1\le q'(s)$ for $s\in[0,1]$, so
\begin{align*}
  \frac{1}{|\Omega|}\int_\Omega h_2^*(u_{n+1}|u^\infty)dx
	&\le \frac{c_S}{q_1|\Omega|}\int_\Omega 
	\frac{q'(u_{n+1})^2}{q(u_{n+1})}|\na u_{n+1}|^2 dx \\
	&= \frac{4c_S}{q_1|\Omega|}\int_\Omega |\na q(u_{n+1})^{1/2}|^2 dx.
\end{align*}
Therefore, \eqref{edi3} yields
$$
  \int_\Omega h_2^*(u_{n+1}(t)|u_{n+1}^\infty)dx
	+ \frac{c_0 q_1}{4c_S}\int_0^t\int_\Omega h_2^*(u_{n+1}(t)|u_{n+1}^\infty)dxds
	\le \int_\Omega h^*(u^0|u^\infty)dx,
$$
and Gronwall's lemma gives
\begin{equation}\label{ei.h2}
  \int_\Omega h_2^*(u_{n+1}(t)|u_{n+1}^\infty)dx 
	\le e^{-c_0 q_1t/(4c_S)}\int_\Omega h^*(u^0|u^\infty)dx.
\end{equation}
The strict positivity of $q'$ implies that the function 
$s\mapsto h_2^*(s|u_{n+1}^\infty)$ is strictly convex. Moreover,
$h_2^*(u_{n+1}^\infty|u_{n+1}^\infty)=0$ and 
$(h_2^*)'(u_{n+1}^\infty|u_{n+1}^\infty)=0$. Therefore, by a Taylor expansion,
$h_2^*(u_{n+1}|u_{n+1}^\infty)\ge (\gamma/2)(u_{n+1}-u_{n+1}^\infty)^2$.
Inserting this inequality in \eqref{ei.h2} gives \eqref{conv.unp1}.

{\em Step 2: Convergence for $(u_i(t))$.} 
We assume that $q(s)\ge q_0>0$ for $s\in[0,1]$.
It follows from the entropy inequality \eqref{edi3} that
$$
  \int_{\Omega}h_1^*(u(t)|u^\infty)dx
	+ c_0q_0\int_{0}^t\int_{\Omega_\eps}\sum_{i=1}^n|\na u_i^{1/2}|^2 dxds
	\le \int_\Omega h^*(u^0|u^\infty)dx, \quad t>0.
$$
We apply the logarithmic Sobolev inequality on bounded domains with constant $c_L>0$
\cite[Lemma 1]{DeFe14},
$$
  \int_{\Omega_\eps}h_1^*(u(t)|u^\infty)dx
	= \sum_{i=1}^n\int_{\Omega}u_i\log\frac{u_i}{u_i^\infty}dx
	\le c_L\sum_{i=1}^n\int_\Omega|\na u_i^{1/2}|^2 dx.
$$
Inserting this inequality into the entropy estimate gives
$$
  \int_{\Omega}h_1^*(u(t)|u^\infty)dx
	+ \frac{c_0q_0}{c_L}\int_{\Omega_\eps}h_1^*(u(t)|u^\infty)dx
	\le \int_\Omega h^*(u^0|u^\infty)dx, \quad t>0,
$$
and then, Gronwall's lemma shows that
$$
  \int_{\Omega}h_1^*(u(t)|u^\infty)dx 
	\le e^{-c_0q_0 t/c_L}\int_\Omega h^*(u^0|u^\infty)dx, \quad t>0.
$$
Finally, since $h_1^\infty(u^\infty|u^\infty)
=|(h_1^*)'(u^\infty,u^\infty)|=0$, and $\pa^2 h_1^*/\pa u_i\pa u_j
=\delta_{ij}/u_i\ge \delta_{ij}$ for $u\in\overline\D$, we obtain 
$h_1^*(u|u^\infty)\ge |u-u^\infty|^2$, which proves estimate \eqref{conv.ui}
and finishes the proof.


\section{Proof of Theorem \ref{thm.unique}}\label{sec.unique}

Let $u=(u_1,\ldots,u_n)$ and $v=(v_1,\ldots,v_n)$ be two bounded weak solutions
to \eqref{eq}-\eqref{bic}. 
Since $p_i\equiv 1$ for all $i=1,\ldots,n$ by assumption, \eqref{eq} becomes
\begin{equation}\label{eq.u}
  \pa_t u_i = \diver\big(q(u_{n+1})\na u_i - u_i\na q(u_{n+1})\big), \quad
	i=1,\ldots,n.
\end{equation}
Summing these equations from $i=1,\ldots,n$, the equation for 
$u_{n+1}=1-\sum_{i=1}^n u_i$ reads as
\begin{equation}\label{eqQ}
  \pa_t u_{n+1} = \diver\big(q(u_{n+1})\na u_{n+1} + (1-u_{n+1})\na q(u_{n+1})\big)
	= \Delta Q(u_{n+1}),
\end{equation}
where $Q(s)=\int_0^s (q(\sigma)+(1-\sigma)q'(\sigma))d\sigma$ for $0\le s\le 1$.
Furthermore, $\na Q(u_{n+1})\cdot\nu=0$ on $\pa\Omega$, $t>0$ and 
$u_{n+1}(0)=u_{n+1}^0:=1-\sum_{i=1}^n u_i^0$, and similar equations holds for
$v_{n+1}$.
Since $Q$ is a nondecreasing function, we can apply first the $H^{-1}$ method to
\eqref{eqQ} to show uniqueness for the $(n+1)$th component, i.e.
$u_{n+1}=v_{n+1}$. Second, we employ the convexity of the entropy 
to prove that $u_i=v_i$ for $i=1,\ldots,n$.

{\em Step 1: Uniqueness for $u_{n+1}$.}
Let $t>0$ and let $\zeta(t)\in H^1(\Omega)$ be the unique solution to 
$$
  -\Delta\zeta(t) = (u_{n+1}-v_{n+1})(t) \quad\mbox{in }\Omega, \quad
	\na\zeta\cdot\nu = 0\quad\mbox{on }\Omega.
$$
We know that $u_{n+1}-v_{n+1}\in L^2(0,T;L^2(\Omega))$. Thus, $t\mapsto \zeta(t)$
is Bochner integrable and $\zeta\in L^2(0,T;H^1(\Omega))$. As 
$\pa_t(u_{n+1}-v_{n+1})\in L^2(0,T;H^1(\Omega)')$, we have even the regularity
$\Delta\pa_t\zeta\in L^2(0,T;H^1(\Omega)')$. Therefore, using \eqref{eqQ},
we obtain for a.e.\ $t>0$,
\begin{align*}
  \frac12\,\frac{d}{dt}\int_\Omega|\na\zeta|^2 dx
	&= \langle -\Delta\pa_t\zeta,\zeta\rangle 
	= \langle\pa_t(u_{n+1}-v_{n+1}),\zeta\rangle \\
	&= -\int_\Omega\na\big(Q(u_{n+1})-Q(v_{n+1})\big)\cdot\na\zeta dx \\
	&= -\int_\Omega\big(Q(u_{n+1})-Q(v_{n+1})\big)(u_{n+1}-v_{n+1})dx.
\end{align*}
Here, $\langle\cdot,\cdot\rangle$ again denotes the duality pairing of
$H^1(\Omega)'$ and $H^1(\Omega)$. The right-hand side is nonpositive since
$Q$ is nondecreasing. This implies that
$$
  \int_\Omega|\na\zeta(t)|^2 dx \le \int_\Omega|\na \zeta(0)|^2 dx, \quad
	t>0.
$$
At time $t=0$, $-\Delta\zeta(0)=(u_{n+1}-v_{n+1})(0)=0$ in $\Omega$,
thus $\na\zeta(0)=0$. Hence, $|\na\zeta(t)|=0$ a.e.\ in $\Omega$, which gives
$(u_{n+1}-v_{n+1})(t)=-\Delta\zeta(t)=0$ in $\Omega$.

{\em Step 2: Uniqueness for $(u_1,\ldots,u_n)$.} 
Let $0<\eps<1$. Similarly as in \cite{Gaj94}, we introduce the distance
\begin{align*}
  & d_\eps(u.v) = \sum_{i=1}^n\int_\Omega\left(\xi_\eps(u_i) + \xi_\eps(v_i)
	- 2\xi_\eps\left(\frac{u_i+v_i}{2}\right)\right)dx, \\
	& \mbox{where }\xi_\eps(s)=(s+\eps)(\log(s+\eps)-1)+1, \ s\ge 0.
\end{align*}
As $\xi_\eps$ is convex, we have 
$\xi_\eps(u_i)+\xi_\eps(v_i)-2\xi_\eps((u_i+v_i)/2)\ge 0$
in $\Omega$ and hence, $d_\eps(u_i,v_i)\ge 0$. We need the regularization
$\eps>0$ since $u_i$ and $v_i$ are only nonnegative and thus, expressions
like $\log((u_i+v_i)/2)$ may be undefined. Since $u_{n+1}=v_{n+1}$ by Step 1,
we may abbreviate $q:=q(u_{n+1})=q(v_{n+1})$. Then, using \eqref{eq.u}, we compute
\begin{align*}
  \frac{d}{dt}d_\eps(u,v) 
	&= \sum_{i=1}^n\bigg(\langle\pa_t u_i,\log(u_i+\eps)\rangle
	+ \langle\pa_t v_i,\log(v_i+\eps)\rangle \\
	&\phantom{xx}{}- \left\langle\pa_t(u_i+v_i),\log\left(\frac{u_i+v_i}{2}
	+\eps\right)\right\rangle\bigg) \\
  &= -\sum_{i=1}^n\int_\Omega\bigg((q\na u_i-u_i\na q)\cdot\frac{\na u_i}{u_i+\eps}
	+ (q\na v_i-v_i\na q)\cdot\frac{\na v_i}{v_i+\eps} \\
	&\phantom{xx}{}- \big(q\na(u_i+v_i) - (u_i+v_i)\na q\big)\cdot
	\frac{\na(u_i+v_i)}{u_i+v_i+2\eps}\bigg)dx.
\end{align*}
Rearranging the terms, we arrive at
\begin{align*}
  \frac{d}{dt}d_\eps(u,v)
  &= -\sum_{i=1}^{n}\int_{\Omega}\left(\frac{|\na u_{i}|^{2}}{u_{i}+\eps} 
	+ \frac{|\na v_{i}|^{2}}{v_{i}+\eps} 
	- \frac{|\na (u_{i}+v_{i})|^{2}}{u_{i}+v_{i}+2\eps}\right) q dx \\
  &\phantom{xx}{} + \sum_{i=1}^{n}\int_{\Omega}\left(\frac{u_{i}}{u_{i}+\eps} 
	- \frac{u_{i}+v_{i}}{u_{i}+v_{i}+2\eps}\right)\na q\cdot\na u_{i}dx \\
  &\phantom{xx}{} + \sum_{i=1}^{n}\int_{\Omega}\left(\frac{v_{i}}{v_{i}+\eps} 
	- \frac{u_{i}+v_{i}}{u_{i}+v_{i}+2\eps}\right)\na q\cdot\na v_{i} dx \\
  &= -4\sum_{i=1}^{n}\int_{\Omega}\left(|\na\sqrt{u_{i}+\eps}|^{2} 
	+ |\na\sqrt{v_{i}+\eps}|^{2} 
	- |\na\sqrt{u_{i}+v_{i}+2\eps}|^{2}\right) q dx \\
  &\phantom{xx}{} + 2\sum_{i=1}^{n}\int_{\Omega}\left(\frac{u_{i}}{u_{i}+\eps} 
	- \frac{u_{i}+v_{i}}{u_{i}+v_{i}+2\eps}\right)\sqrt{q}\na\sqrt{q}\cdot
	\na u_{i} dx \\
  &\phantom{xx}{} + 2\sum_{i=1}^{n}\int_{\Omega}\left(\frac{v_{i}}{v_{i}+\eps} 
	- \frac{u_{i}+v_{i}}{u_{i}+v_{i}+2\eps}\right)\sqrt{q}\na\sqrt{q}\cdot\na v_{i}dx.
\end{align*}
Now, we apply Lemma \ref{lem.fisher} with $d\mu=qdx$ and $f=u_i+\eps$,
$g=v_i+\eps$, showing that the first integral on the right-hand side is
nonnegative. We observe that $d_\eps(u(0),v(0))=0$ as $u$ and $v$ have the same
initial data. Thus, integrating the above expression in time, we obtain
\begin{align}
  d_\eps(u(t),v(t)) &\le 2\sum_{i=1}^{n}\int_0^t\int_{\Omega}
	\left(\frac{u_{i}}{u_{i}+\eps} 
	- \frac{u_{i}+v_{i}}{u_{i}+v_{i}+2\eps}\right)\sqrt{q}\na\sqrt{q}\cdot
	\na u_{i} dx \label{aux4} \\
  &\phantom{xx}{} + 2\sum_{i=1}^{n}\int_0^t\int_{\Omega}\left(\frac{v_{i}}{v_{i}+\eps} 
	- \frac{u_{i}+v_{i}}{u_{i}+v_{i}+2\eps}\right)\sqrt{q}\na\sqrt{q}\cdot\na v_{i}dx.
	\nonumber
\end{align}
Since $\na\sqrt{q}$, $\sqrt{q}\na u_{i}$, $\sqrt{q}\na v_{i}\in 
L^{2}(0,T; H^{1}(\Omega))$ and
$$
  \left| \frac{u_{i}}{u_{i}+\eps} \right|\leq 1,\quad 
	\left| \frac{v_{i}}{v_{i}+\eps} \right|\leq 1,\quad
  \left| \frac{u_{i}+v_{i}}{u_{i}+v_{i}+2\eps} \right|\leq 1,
$$
the dominated convergence implies that the right-hand side of \eqref{aux4}
tends to zero as $\eps\to 0$. From the nonnegativity of $d_\eps$ we deduce that
$d_\eps(u(t),v(t))\to 0$ as $\eps\to 0$, which means that
\begin{equation}\label{aux5}
  \xi_\eps(u_i) + \xi_\eps(v_i) - 2\xi_\eps\left(\frac{u_i+v_i}{2}\right)
	\to 0 \quad \mbox{as }\eps\to 0 \quad\mbox{a.e. in }\Omega\times(0,\infty).
\end{equation}
According to Taylor's formula, there are functions $\theta_\eps$, $\eta_\eps:
\Omega\times(0,\infty)$ such that 
\begin{align*}
  &\xi_\eps(u_i) = \xi_\eps\left(\frac{u_i+v_i}{2}+\frac{u_i-v_i}{2}\right) \\
	&= \xi_{\eps}\left(\frac{u_{i}+v_{i}}{2}\right) 
	+ \xi_{\eps}'\left(\frac{u_{i}+v_{i}}{2}\right)\frac{u_{i}-v_{i}}{2} 
  + \frac{1}{2}\xi_{\eps}''\left(\theta_{\eps}\frac{u_{i}+v_{i}}{2} 
	+ (1-\theta_{\eps})u_{i}\right)\left(\frac{u_{i}-v_{i}}{2}\right)^{2}, \\
  &\xi_{\eps}(v_{i}) = \xi_{\eps}\left( \frac{u_{i}+v_{i}}{2} 
	- \frac{u_{i}-v_{i}}{2} \right) \\
  &= \xi_{\eps}\left(\frac{u_{i}+v_{i}}{2}\right) 
	- \xi_{\eps}'\left(\frac{u_{i}+v_{i}}{2}\right)\frac{u_{i}-v_{i}}{2} 
  + \frac{1}{2}\xi_{\eps}''\left(\eta_{\eps}\frac{u_{i}+v_{i}}{2} 
	+ (1-\eta_{\eps})v_{i}\right)\left(\frac{u_{i}-v_{i}}{2}\right)^{2}.
\end{align*}
Adding these identities and employing the estimate $\xi_\eps''(s)=(s+\eps)^{-1}
\ge 1/2$ for $0\le s\le 1$, we infer that
$$
  \xi_{\eps} (u_{i}) + \xi_{\eps}(v_{i}) 
	- 2 \xi_{\eps}\left( \frac{u_{i}+v_{i}}{2}\right) 
	\ge\frac18(u_i-v_i)^2.
$$
This estimate and \eqref{aux5} prove that $u_i=v_i$ in $\Omega\times(0,\infty)$
for $i=1,\ldots,n$.


\section{Extensions}\label{sec.ext}

In this section, we discuss some extensions of the diffusion system \eqref{eq}.

{\em Reaction terms.} Cross-diffusion systems with reaction terms,
\begin{equation}\label{eq.f}
  \pa_t u - \diver(A(u)\na u) = f(u) \quad\mbox{in }\Omega,\ t>0,
\end{equation}
can be treated similarly as in \cite{Jue14}. More precisely, if
there is a constant $c_f>0$ such that $f(u)\cdot h'(u)\le c_f(1+h(u))$ for 
all $u\in\D$, then there exists a global weak solution to \eqref{bic} and
\eqref{eq.f}. The proof proceeds as for Theorem \ref{thm.ex}, where 
the right-hand side of the entropy inequality \eqref{edi} has to be replaced by
$$
  \int_\Omega h(u^0)dx + \tau\int_\Omega f(u)\cdot h'(u)dx
	\le \int_\Omega h(u^0)dx + \tau c_f\int_\Omega(1+h(u))dx.
$$
Then, for sufficiently small $\tau>0$, the integral $\tau c_f\int_\Omega h(u)dx$
can be absorbed by the left-hand side of \eqref{edi}.
For instance, reaction terms of Lotka-Volterra type
$$
  f_i(u) = u_i\bigg(1-\sum_{j=1}^n s_{ij}u_j\bigg), \quad i=1,\ldots,n, \ 
	s_{ij}\ge 0,
$$
are admissible. The large-time behavior result is valid only under
an additional condition on $f(u)$, namely $f(u)\cdot h'(u)\le 0$ for $u\in\D$.
If we suppose conservation of the ``total mass'', i.e.\
$\sum_{i=1}^n f_i(u)=0$, the $H^{-1}$ method allows us to prove uniqueness
for $u_{n+1}$. Uniqueness for the remaining components $u_i$ follows if there
exists $C>0$ such that for all $u_i$ and $v_i$, 
$$
  \sum_{i=1}^n\left(f_i(u)\log\frac{2u_i}{u_i+v_i}
	+ f_i(v)\log\frac{2v_i}{u_i+v_i}\right)
	\le C\sum_{i=1}^n\left(\xi(u_i)+\xi(v_i)-2\xi\left(\frac{u_i+v_i}{2}\right)\right),
$$
where $\xi(s)=s(\log s-1)+1$ (see the proof of Theorem \ref{thm.unique}). More
general conditions on $f(u)$ can be found in \cite{GaSk04}.

{\em Drift terms.} In the presence of environmental or electric potentials
or of chemotactic signal concentrations, 
the diffusion system contains additional drift terms,
\begin{equation}\label{drift}
  \pa_t u - \diver(A(u)\na u + D(u)\na\phi) = 0\quad\mbox{in }\Omega,\ t>0,
\end{equation}
where $D(u)=(D_{ij}(u))$ is an $n\times n$ matrix and the $i$th component of 
$D(u)\na\phi$
is given by $\sum_{j=1}^n D_{ij}(u)\na\phi_j$, where $\phi_j=\phi_j(x)$ 
is some potential.
Assume that $h$ is such that 
$\na u:h''(u)A(u)\na u \ge \sum_{i=1}^n g_i(u)|\na u_i|^2$ for some
nonnegative functions $g_i(u)$. Then, using the test function $h'(u)$ in the
weak formulation of \eqref{drift}, we compute
\begin{align*}
  \frac{d}{dt}\int_\Omega h(u)dx &= -\int_\Omega\na u:h''(u)A(u)\na u dx
	- \int_\Omega \na u:h''(u)D(u)\na\phi dx \\
	&\le -\frac12\sum_{i=1}^n\int_\Omega g_i(u)|\na u_i|^2dx
	+ \frac12\sum_{k=1}^n\int_\Omega G_k(u)|\na\phi_k|^2dx,
\end{align*}
where we employed the Cauchy-Schwarz inequality and have set
$G_k(u)=\sum_{i,j=1}^n g_i(u)^{-1}H_{ij}^2$ $D_{jk}(u)^2$ with $H=h''(u)$.
Thus, if $\na\phi_i$ is bounded in $L^2$ and $G_k(u)$ in $L^\infty$, we achieve
some gradient estimates, which are the basis for the existence analysis.
An example is the ion-transport model \cite{BSW12}
$$
  A_{ij}(u) = u_i\quad \mbox{for }i\neq j, \quad
	A_{ii}(u) = u_i+u_{n+1}, \quad D_{ij}(u) = u_i u_{n+1}\delta_{ij}
$$
for $i,j=1,\ldots,n$. The entropy density can be defined by
$$
  h(u) = \sum_{i=1}^{n}\big(u_i(\log u_i-1) + u_i\phi_i\big)
	+ u_{n+1}(\log u_{n+1}-1).
$$
Then the Hessian $h''(u)$ does not depend on $\phi_i$.
A formal computation, using the Cauchy-Schwarz inequality and 
the identity $\sum_{i=1}^n\na u_i=-\na u_{n+1}$, gives 
\begin{align*}
  \frac{d}{dt} & \int_\Omega h(u)dx 
	= -\sum_{i=1}^n\int_\Omega
	u_iu_{n+1}\left|\na\left(\log\frac{u_i}{u_{n+1}}+\phi_i\right)\right|^2 dx \\
  &\le -\sum_{i=1}^n\int_\Omega
	u_iu_{n+1}\left(\frac12\left|\na\log\frac{u_i}{u_{n+1}}\right|^2
	- |\na\phi_i|^2\right)dx \\
	&= -\sum_{i=1}^n\int_\Omega\big(2u_{n+1}|\na u_i^{1/2}|^2 + |\na u_{n+1}|^2
	+ 2|\na u_{n+1}^{1/2}|^2\big)dx 
	+ \sum_{i=1}^n\int_\Omega u_iu_{n+1}|\na\phi_i|^2 dx.
\end{align*}
As $q(s)=s$ in this model, we find the same estimates as in the proof of Theorem 
\ref{thm.ex} (also see \cite[Section 3.2]{BDPS10}). This shows that our
strategy can be adapted to cross-diffusion systems with drift.

{\em Other diffusion coefficients.} Our main assumption on the transition
rates is that they are given by the product of $p_i(u)$ and $q_i(u_{n+1})$
(see Appendix \ref{sec.deriv}). Also other choices are possible. An example
is the diffusion system of \cite{Pai09}, which is derived from a stochastic
lattice model
by assuming that the transition rates are given by $p_i(u)+q_i(u_{n+1})$
for some special functions $p_i$ and $q_i$. 
The diffusion matrix has the structure
$$
  A(u) = \begin{pmatrix}
	\alpha_1(1-u_2) + u_2 & (\alpha_1-1)u_1 \\
	(\alpha_2-1)u_2 & \alpha_2(1-u_1) + u_1
	\end{pmatrix},
$$
where $\alpha_1$, $\alpha_2>0$. The corresponding diffusion system
possesses the entropy density
$$
  h(u) = \sum_{i=1}^2 u_i(\log u_i-1) + (1-u_1-u_2)(\log(1-u_1-u_2)-1), \quad
	u=(u_1,u_2)\in\D,
$$
and the new diffusion matrix $B=h''(u)^{-1}A(u)$, given by
$$
  B = \begin{pmatrix}
	(\alpha_1(1-u_1-u_2)+u_2)u_1 & -u_1u_2 \\
	-u_1u_2 & (\alpha_2(1-u_1-u_2)+u_1)u_2
	\end{pmatrix},
$$
is symmetric and positive semi-definite on $\overline{\D}$. 
For our analysis, we need bounds from $h''(u)A(u)$ (see Lemma \ref{lem.HA}),
which are less obvious since
\begin{align*}
  \na u_1^\top h''(u)A(u)\na u_2 
	&= \frac{a_1|(1-u_2)\na u_1 + u_1\na u_2|^2}{u_1(1-u_1-u_2)}
	+ \frac{a_2|u_2\na u_1 + (1-u_1)\na u_2|^2}{u_2(1-u_1-u_2)} \\
	&\phantom{xx}{}+ 4|\na\sqrt{u_1u_2}|^2,
\end{align*}
only yielding an $L^2$ bound for $\na\sqrt{u_1u_2}$ in $L^2$.


\begin{appendix}
\section{Formal derivation of the $n$-species population model}\label{sec.deriv}

We derive formally the cross-diffusion system \eqref{eq} from a master equation
for a discrete-space random walk in the diffusion limit. 
We consider random walks on a one-dimensional lattice only, since the
derivation can be extended in a straightforward manner to the higher-dimensional
situation. The lattice is given by cells $x_j$ ($j\in{\mathbb Z}$) with the
uniform cell distance $h=x_j-x_{j-1}>0$. 
The proportions of the $i$th population in the $j$th cell at time $t>0$ is denoted 
by $u_i(x_j)=u_i(x_j,t)$. The species move from the $j$th cell
into the neighboring cells $j\pm 1$ with the transition rates $T^{j,\pm}_i$.
The master equations are given by
$$
  \pa_t u_i(x_j) = T_i^{j-1,+}u_i(x_{j-1}) + T_i^{j+1,-}u_i(x_{j+1})
	- (T_i^{j,+}+T_j^{i,-})u_i(x_j), \quad i=1,\ldots,n,
$$
and the transition rates are defined as 
\begin{equation}\label{trans}
  T_i^{j,\pm} = \sigma_0 p_i(u(x_j))q_i(u_{n+1}(x_j)), \quad 
	u_{n+1}(x_j)= 1-\sum_{k=1}^nu_k(x_j),
\end{equation}
where $u=(u_1,\ldots,u_n)$. The quantities $p_i(u(x_j))$ and $q_i(u_{n+1}(x_{j\pm 1}))$
measure the tendency of the species $i$ to leave the $j$th cell or to move
into the $j$th cell from one of the neighboring cells, respectively. 
More precisely, $u_i(x_j)$ denotes a volume fraction of occupancy and $u_{n+1}$ the
volume fraction not occupied by the species. Our assumption is that the
transition rates, measuring the occupancy and the non-occupancy, separate,
resulting in the product of $p_i$ and $q_i$. Other choices are possible
(see \cite{Pai09} for an example), but the analytical treatment of the
corresponding diffusion systems is not obvious.

For the derivation of the diffusion model, it is convenient to introduce the
following abbreviations:
\begin{align*}
  p^{j}_{i} &= p_{i}(u_{1}(x_{j}),\ldots,u_{n}(x_{j}) ),\quad 
	q^{j}_{i} = q_{i}(u_{n+1}(x_{j})), \\
  \pa_{k}p^{j}_{i} &= \frac{\pa p_{i}}{\pa u_{k}}(u_{1}(x_{j}),\ldots,u_{n}(x_{j})),
	\quad \pa q^{j}_{i} = q_{i}'(u_{n+1}(x_{j})).
\end{align*}
Thus, we can rewrite the master equation as
\begin{equation}\label{app.me}
  \sigma_0^{-1}\pa_t u_i^j = q_i^j(p_i^{j-1}u_i^{j-1} + p_i^{j+1}u_i^{j+1})
	- p_i^ju_i^j(q_i^{j+1}+q_i^{j-1}).
\end{equation}
Set $D=\pa_x$. 
We compute the Taylor expansions of $p_i$ and $q_i$ ($i=1,\ldots,n$) and
replace $u_k^{j\pm 1}-u_k^j$ by the Taylor expansion 
$\pm hDu_k^j + \frac12 h^2 D^2u_i^j + O(h^3)$. Then, collecting all terms
up to order $O(h^2)$, we arrive at
\begin{align*}
  p_i^{j\pm 1} &= p_i^j + h\sum_{k=1}^n\pa_k p_i^j Du_k^j
	+ \frac{h^2}{2}\left(\sum_{k=1}^n\pa_k p_i^jD^2u_k^j
	+ \sum_{k,\ell=1}^n\pa_{k\ell}^2 p_i^j Du_k^j Du_\ell^j\right) + O(h^3), \\
	q_i^{j\pm 1} &= q_i^j \pm h\pa p_i^j Du_{n+1}^j
	+ \frac{h^2}{2}\big(\pa q_i^j D^2 u_{n+1}^j + \pa^2 q_i^j(Du_{n+1}^j)^2\big) 
	+ O(h^3) \\
	&= q_i^j \mp h\pa q_i^j\sum_{k=1}^n Du_k^j 
	+ \frac{h^2}{2}\left(-\pa q_i^j\sum_{k=1}^n D^2 u_k^j
	+ \pa^2 q_i^j\sum_{k,\ell=1}^n Du_k^j Du_\ell^j\right) + O(h^3).
\end{align*}
In the last step, we have used $u_{n+1}=1-\sum_{k=1}^n u_k$.
We insert these expressions into \eqref{app.me} and rearrange the terms.
It turns out that the terms of order $O(1)$ and $O(h)$ cancel, and we end up with
\begin{align*}
  \sigma_0^{-1}h^{-2}\pa_t u_i^j
	&= \sum_{k=1}^n D^2 u_k^j(q_i^j p_i^j\delta_{ik} + q_i^j u_i^j \pa_k p_i^j
	+ p_i^j u_i^j\pa q_i^j) \\
	&\phantom{xx}{}+ \sum_{k,\ell=1}^n Du_k^j Du_\ell^j(2q_i^j\pa_k p_i^j\delta_{i\ell}
	+ q_i^j u_i^j\pa_{k\ell}^2 p_i^j - p_i^j u_i^j\pa^2 q_i^j).
\end{align*}
We choose $\sigma_0=h^{-2}$ and pass to the limit $h\to 0$:
\begin{align*}
  \pa_t u_i &= \sum_{k=1}^n D^2 u_k\left(q_i p_i\delta_{ik}
	+ q_iu_i\frac{\pa p_i}{\pa u_k} + p_iu_i q_i'\right) \\
	&\phantom{xx}{}+ \sum_{k,\ell=1}^n Du_k Du_\ell\left(2q_i\frac{\pa p_i}{\pa u_k}
	\delta_{i\ell} + q_iu_i\frac{\pa^2 p_i}{\pa u_k\pa u_\ell}
  - p_iu_iq_i''\right). 
\end{align*}
A lenghty but straightforward computation shows that the last sum equals
$$
   \sum_{k=1}^n Du_k D\left(q_ip_i\delta_{ik} + q_iu_i\frac{\pa p_i}{\pa u_k}
	+ p_iu_iq_i'\right),
$$
and we end up with 
$$
  \pa_t u_i = D\sum_{k=1}^n Du_k\left(q_i p_i\delta_{ik}
	+ q_iu_i\frac{\pa p_i}{\pa u_k} + p_iu_i q_i'\right),
$$
which is the one-dimensional version of \eqref{eq}.

\end{appendix}



\begin{thebibliography}{11}
\bibitem{Ama89} H.~Amann. Dynamic theory of quasilinear parabolic systems. III. 
Global existence. {\em Math. Z.} 202 (1989), 219-250.

\bibitem{AMTU01} A.~Arnold, P.~Markowich, G.~Toscani, and A.~Unterreiter. 
On convex Sobolev inequalities and the rate of convergence to equilibrium for 
Fokker-Planck type equations. {\em Commun. Part. Diff. Eqs.} 26 (2001), 43-100. 

\bibitem{BGL14} D.~Bakry, I.~Gentil, and M.~Ledoux. {\em Analysis and Geometry
of Markov Diffusion Operators}. Springer, Cham, 2014.

\bibitem{BLMP09} M.~Bendahmane, T.~Lepoutre, A.~Marrocco, and B.~Perthame.
Conservative cross diffusions and pattern formation through relaxation.
{\em J. Math. Pures Appl.} 92 (2009), 651-667.

\bibitem{Bre11} H.~Br\'ezis. {\em Functional analysis, Sobolev spaces and partial 
differential equations}. Springer, New York, 2011.

\bibitem{BrCh12} M.~Bruna and S.~J.~Chapman. Diffusion of multiple species with
excluded-volume effects. {\em J. Chem. Phys.} 137 (2012), 204116, 16 pages.

\bibitem{BDPS10} M.~Burger, M.~Di Francesco, J.-F.~Pietschmann, and B.~Schlake.
Nonlinear cross-diffusion with size exclusion. {\em SIAM J. Math. Anal.} 42 (2010),
2842-2871.

\bibitem{BSW12} M.~Burger, B.~Schlake, and W.-T.~Wolfram. Nonlinear 
Poisson-Nernst-Planck equations for ion flux through confined geometries. 
{\em Nonlinearity} 25 (2012), 961-990.

\bibitem{ChJu04} L.~Chen and A.~J\"ungel. Analysis of a multi-dimensional
parabolic population model with strong cross-diffusion.
{\em SIAM J. Math. Anal.} 36 (2004), 301-322.

\bibitem{ChJu06} L.~Chen and A.~J\"ungel. Analysis of a parabolic cross-diffusion 
population model without self-diffusion. {\em J. Diff. Eqs.} 224 (2006), 39-59.

\bibitem{CJL14} X.~Chen, A.~J\"ungel, and J.-G.~Liu. A note on 
Aubin-Lions-Dubinskii lemmas. {\em Acta Appl. Math.} 133 (2014). 33-43.

\bibitem{DeFe14} L.~Desvillettes and K.~Fellner, Exponential convergence to 
equilibrium for a nonlinear reaction-diffusion systems arising in reversible
chemistry. In: {\em System Modelling and Optimization}, Proceedings of the
IFIP TC 7 Conference 2013. {\em Adv. Inform. Commun. Techn.} 443 (2014), 
96-104.


\bibitem{DLM14} L.~Desvillettes, T.~Lepoutre, and A.~Moussa. Entropy,
duality and cross diffusion. {\em SIAM J. Math. Anal.} 46 (2014), 820-853.

\bibitem{DLMT14}  L.~Desvillettes, T.~Lepoutre, A.~Moussa, and A.~Trescases.
On the entropic structure of reaction-cross diffusion systems. Preprint, 2014.
{\tt arXiv:1410.7377}.

\bibitem{DrJu12} M.~Dreher and A.~J\"ungel. Compact families of piecewise constant 
functions in $L^p(0,T;B)$. {\em Nonlin. Anal.} 75 (2012), 3072-3077.

\bibitem{Gaj94} H.~Gajewski. On a variant of monotonicity and its application 
to differential equations. {\em Nonlin. Anal. TMA} 22 (1994), 73-80.

\bibitem{GaSk04} H.~Gajewski and I.~Skrypnik. On the uniqueness problem for
nonlinear parabolic equations. {\em Discr. Cont. Dynam. Sys.} 10 (2004), 315-336.

\bibitem{GGJ03} G.~Galiano, M.~Garz\'on, and A.~J\"ungel. Semi-discretization in 
time and numerical convergence of solutions of a nonlinear cross-diffusion 
population model. {\em Numer. Math.} 93 (2003), 655-673.


\bibitem{Jue14} A.~J\"ungel. The boundedness-by-entropy principle for cross-diffusion
systems. Preprint, 2014. \newline {\tt arXiv:1403:5419}.

\bibitem{Kim84} J.~Kim. Smooth solutions to a quasi-linear system of diffusion 
equations for a certain population model. {\em Nonlin. Anal.} 8 (1984), 1121-1144.

\bibitem{Le02} D.~Le. Cross diffusion systems in $n$ spatial dimensional domains.
{\em Indiana Univ. Math. J.} 51 (2002), 625-643.

\bibitem{Mou14} A.~Moussa. Some variants of the classical Aubin-Lions Lemma.
Preprint, 2014. {\tt arXiv:1401.7231v3}.

\bibitem{PaUn95} F.~Pacard and A.~Unterreiter. A variational analysis of the
thermal equilibrium state of charged quantum fluids. {\em Commun. Part. Diff. Eqs.}
20 (1995), 885-900.

\bibitem{Pai09} K.~Painter. Continuous models for cell migration in tissues
and applications to cell sorting via differential chemotaxis. 
{\em Bull. Math. Biol.} 71 (2009), 1117-1147.

\bibitem{SKT79} N.~Shigesada, K.~Kawasaki, and E.~Teramoto. Spatial segregation 
of interacting species. {\em J. Theor. Biol.} 79 (1979), 83-99.

\bibitem{Sim87} J.~Simon. Compact sets in the space $L^p(0,T;B)$.
{\em Ann. Math. Pura. Appl.} 146 (1987), 65-96.

\bibitem{SLH09} M.~Simpson, K.~Landman, and B.~Hughes. Multi-species simple
exclusion processes. {\em Physica A} 388 (2009), 399-406.

\end{thebibliography}
\end{document}